\documentclass{amsart}
\usepackage{comment}
\usepackage{subcaption}
\usepackage[english]{babel}
\usepackage{amssymb}
\usepackage{enumerate, xspace}
\usepackage{dsfont}
\usepackage{amsfonts}
\usepackage{xcolor}
\usepackage{graphicx}
\graphicspath{ {./} }
\usepackage{amsmath}
\usepackage[normalem]{ulem}
\usepackage{graphicx}
\usepackage[T1]{fontenc}
\usepackage[utf8]{inputenc} 
\usepackage{bbm}
\usepackage[colorlinks=true,linkcolor=violet,citecolor=blue,urlcolor=blue]{hyperref}

\newtheorem{theorem}{Theorem}

\newtheorem{mainthm}{Theorem}

\newtheorem{corollary}[theorem]{Corollary}

\newtheorem{defin}[theorem]{Definition}

\newtheorem{remark}[theorem]{Remark}

\newtheorem{lemma}{Lemma}[section]
\newtheorem{proposition}{Proposition}[section]

\newcommand\cA{{\mathcal A}}
\newcommand\cB{{\mathcal B}}
\newcommand\cC{{\mathcal C}}

\newcommand\cF{{\mathcal F}}

\newcommand\cH{{\mathcal H}}
\newcommand\cI{{\mathcal I}}
\newcommand\cJ{{\mathcal J}}

\newcommand\cL{{\mathcal L}}

\newcommand\cT{{\mathcal T}}
\newcommand\cU{{\mathcal U}}

\newcommand\bC{{\mathbb C}}

\newcommand\bN{{\mathbb N}}

\newcommand\bP{{\mathbb P}}

\newcommand\bR{{\mathbb R}}
\newcommand\bT{{\mathbb T}}

\newcommand\bZ{{\mathbb Z}}

\newcommand\fB{{\mathfrak B}}

\newcommand\ve{\varepsilon}
\newcommand\eps{\ve}

\newcommand{\norm}[1]{\left\Vert#1\right\Vert}

\begin{document}
\title{Annealed Ruelle-Pollicott Resonances}

\author{Sakshi Jain$^{1}$}
\author{Maxence Phalempin$^{2}$}

\maketitle

\begingroup
\renewcommand{\thefootnote}{}
\footnotetext{
\begin{itemize}

\item[$^{1}$] School of Mathematics and Physics, University of Queensland, St Lucia, QLD 4072, Australia. \texttt{sakshi.jain@uq.edu.au}
\item[$^{2}$] School of Mathematics and Statistics, University of New South Wales Sydney, Sydney, NSW 2052, Australia. \texttt{m.phalempin@unsw.edu.au}
\end{itemize}
}
\endgroup
\begin{abstract}
   We adapt the theory of Ruelle–Pollicott resonances to annealed random dynamical systems generated by independent and identically distributed families of maps. Introducing annealed transfer and Koopman operators, we define resonances as elements of the point spectrum of the associated operators and establish a decorrelation formula relating these resonances to the asymptotic decay of annealed correlations. We then study several classes of systems for which the theory can be made explicit. First, we consider a family of piecewise expanding Markov maps of the interval, we construct Banach spaces adapted to the dynamics and obtain a complete description of the annealed resonance spectrum. Then we investigate an inverse spectral problem, proving realisability results for prescribed collections of complex numbers as resonances of suitably constructed annealed dynamical systems. 

\end{abstract}
\section{Introduction}
One of the central aims of ergodic theory and statistical mechanics is to quantify how deterministic or random dynamical systems lose memory of their initial states. 

Consider a map \(T\) on a  measure space \(X\) (with added structure as required), preserving a probability measure \(\mu\). Consider two bounded functions \(f\) and \(g\) on \(X\). A standard feature that encapsulates a lot of information on its probabilistic behaviour is the correlation function between observables, given by 
 \[
 \int f\cdot g\circ T^n - \left(\int f\ d\mu\right)\cdot \left(\int g\ d\mu\right), 
 \]
 which encodes fundamental statistical properties like the decay of this correlation function with \(n\) corresponds to mixing and statistical stabilization. The rate of this decay provides information about the speed of mixing, for example, exponential decay implies strong mixing and rapid statistical stabilization, while slower, polynomial rates signal long-range memory effects. Such correlation asymptotics play a decisive role in limit theorems, linear response, and stability analysis with a well developed general theory (see, for example, \cite{Liverani95,Dolgopyat98,KH96,Keller98,Baladi2000,Baladi16}).
 
Beyond mere rates, it is sometimes possible to obtain the next few terms in their asymptotic expansion, in terms of the Ruelle–Pollicott spectrum (or Ruelle–Pollicott resonances) of the map. They were first introduced in \cite{Ruelle86}, and were rigorously studied for the first time in \cite{Pollicott86} for hyperbolic dynamical systems. The literature has been flourishing ever since. Since these foundational contributions, a comprehensive functional-analytic framework has been developed to rigorously define and study resonances via transfer operators acting on suitably adapted Banach spaces. In particular, the works of Baladi \cite{Baladi2000,Baladi16}, Liverani \cite{Liverani95} and with Gouëzel \cite{GL2006}, have established the existence of discrete spectra for large classes of uniformly hyperbolic systems and clarified their relation to decay of correlations and statistical limit laws.

Resonances play a central role in quantifying the statistical properties of chaotic systems. They govern decay of correlations and thereby mixing rates, provide error terms in statistical limit laws, and they appear in linear response theory, linking them to sensitivity of invariant measures on initial conditions. Beyond pure dynamics, resonances have deep connections to scattering theory in quantum chaos and to non-equilibrium statistical mechanics, where they describe relaxation to equilibrium. Numerically, approximations of resonances have found applications in fluid mixing, climate dynamics, and molecular simulations (see \cite{CNKMG14, GL17}). This broad scope underlines the importance of locating resonances in new settings, such as annealed random systems.

There are various ways of defining resonances but we will go by the following definition taken from \cite{FGL19} (replacing the space by \(\cC^\infty\), which in the original version is \(\cC_b\), the space of bounded functions).

\begin{defin}\cite{FGL19}
    Let \(T\) be a map on a space \(X\), preserving a probability measure \(\mu\).
Consider the space of functions \(\cC^\infty\) on \(X\). Let \(I\) be a finite or countable set, let \(\Lambda = (\lambda_i)_{i\in I}\) be a set of complex numbers with \(|\lambda_i|\in (0,1]\) such that for any \(\ve>0\) there are only finitely many \(i\) with \(|\lambda_i|\geq\ve\), and let \((N_i)_{i\in I}\) be nonnegative integers. We say that \(T\) has a full Ruelle-Pollicott spectrum \((\lambda_i)_{i\in I}\) with Jordan blocks dimension \((N_i)_{i\in I}\) on the space of functions \(\cC^\infty\) if, for any \(f,g \in \cC^\infty\) and for any \(\ve > 0\), there is an asymptotic
expansion
\[
\int f \cdot g\circ T^n\ d\mu = \sum\limits_{|\lambda_i|\geq\ve}\sum\limits_{j\leq N_i}\lambda_i^n n^j c_{i,j}(f,g) + o(\ve^n),
\]
where \(c_{i,j}(f,g)\) are bilinear functions of \(f\) and \(g\), that we suppose finite rank but non-zero.
\end{defin}
In other words, there is an asymptotic expansion for the correlations of functions
in \(\cC^\infty\), up to an arbitrarily small exponential error. With this definition, it is clear that the Ruelle-Pollicott spectrum is an intrinsic object, only depending on \(T, \mu\) and the space of functions \(\cC^\infty\). 

As mentioned above, these resonances can also be studied via transfer operator theory, namely these are also the complex numbers associated with the spectrum of the transfer operator acting on suitable Banach or Hilbert spaces. They appear as poles of the Laplace transform of the correlation functions, for more details refer to Definition~\ref{def:main}. Using the transfer operator techniques, the existence of resonances has been studied extensively in the last few decades for various systems including expanding, piecewise expanding \cite{Baladi2000,BKL22}, hyperbolic systems \cite{Liverani95, GL2006, Baladi16}, pseudo-Anosov systems \cite{FGL19}, real analytic uniformly hyperbolic systems (expanding circle maps and uniformly hyperbolic maps on \(\bT^2\)) in \cite{FN2006} etc.  More recently, a pseudospectral approach was introduced in \cite{BNT2025}, for the rigorous, computer-assisted estimation of resonances, providing regions where resonances must exist and precluding the presence of resonances elsewhere. Their approach enables to computationally enclose the resonances of concerned system within small discs. We would also like to point out the fact that it is not always possible to obtain resonances to arbitrary precision for example \cite{BCJ23, BCC25} proves that the essential spectrum for non-Markov maps is persistent causing an obstruction in locating the entire point spectrum.

Many systems of current interest, however, are not purely deterministic but involve randomness. Random dynamical systems arise naturally in contexts such as stochastic perturbations of deterministic models, physical systems driven by external noise, and dynamics in random media. Two different frameworks are typically distinguished: the quenched setting, where one fixes a random realization and studies the resulting system, and the annealed setting, where one averages over randomness at each iteration, producing an averaged transfer operator. While the quenched framework is conceptually closer to individual realizations, the annealed framework has the advantage of tractability and is more directly connected to physically measurable averaged quantities.

In i.i.d.\ random systems, correlation functions can still be defined in terms of the averaged evolution, and they remain key objects of study. Yet, in contrast to deterministic systems, the spectral theory of annealed transfer operators remains comparatively undeveloped. In particular, the precise location of resonances for these operators has not been systematically explored. Understanding these resonances is crucial: they govern the asymptotic behavior of annealed correlations, describe how randomness modifies or suppresses memory, and can reveal robustness or fragility of statistical properties under randomization.

The goal of the present work is to contribute to this spectral theory. We develop a framework for locating Ruelle–Pollicott resonances for annealed operator, thereby linking decay of correlations to explicit spectral data. We illustrate the general results with concrete examples, showing how resonances can be computed, and how their structure compares with the deterministic case.

We also consider an inverse problem, which refers to asking the question: given a collection of complex numbers in unit disc, does there exist a family of dynamical systems realising those as the resonances of the corresponding annealed operator. The answer turns out to be affirmative, we consider a family of hyperbolic Blaschke products on \(\bT^2\) such that the annealed operator realises the average of the complex numbers (each of which is  a resonance of one of the Blaschke products in the family)  as its resonance. \\
It is however not true in general that the resonances of annealed operator are the average of resonances of individual systems. We support this claim with Example~\ref{sec:eg}.

Taken together, our results establish that the spectral viewpoint extends fruitfully to annealed operator, offering new insight into how randomization interacts with the statistical properties of dynamics. Some future questions that the authors will be interested to address are the interplay with quenched resonances, universality of spectral features, and extensions to continuous-time random flows.

\subsection*{Overview and outlook}
The rest of the article is organised as follows: In Section~\ref{sec:results}, we recall the definition of Ruelle-Pollicott resonances for deterministic systems. We then describe the random system in the i.i.d.\ case with the annealed transfer and Koopman operator, asserting the well-definedness of the same, followed by the definition of decorrelation formula for the annealed operator, finally stating our main results, namely Theorem~\ref{thm:inverse} and Theorem~\ref{mainthm:B} along with Proposition~\ref{prop:gencomp} and Proposition~\ref{propgenquasi}. In Section~\ref{sec:piecexpand}, we describe a class of piecewise expanding Markov systems on the unit interval, where the key is to find a family of Banach spaces which is well adapted to any of these Markov maps. We give a full description of Ruelle-Pollicott resonances for the i.i.d.\ annealed operator associated to such a family. Section~\ref{secblaschke} consists of two parts. First we study an inverse problem, that is, whether a chosen family of complex numbers in the open unit disc (or their average) can be realised as the resonances of an annealed operator. Second part is an example of stochastic differential equation which realises the associated transfer operator as an annealed operator and we numerically find the resonances. Finally, we have added an appendix which includes the proofs of various lemmata and propositions used within the article.

\section{General framework and results}\label{sec:results}

\subsection{Ruelle-Pollicott resonances in deterministic dynamical systems.}\label{subsecdeter} 

For a dynamical system $(X,T,m)$ where $X$ stands for a metric space, $m$ a reference measure, and $T:X\to X$ a (non-singular) transformation. The Koopman operator on $L^1(m)$ is classically defined as $U f:= f\circ T$ and the transfer operator $L$ as the adjoint of the Koopman operator $U$ : for $f\in L^1(m)$ and $g\in L^\infty(m)$, 

\begin{align}\label{eqoptrans0}
    \int L f gdm=\int  f g\circ T dm.
\end{align}

Given a Banach space $\cB \subset L^1(m)$, one can define the \textbf{Ruelle-Pollicott resonances} of the dynamical system $(X,T,m)$ over $\cB$ as the generalized eigenspace corresponding to the isolated eigenvalues.\\
 
Whenever the essential spectrum $\sigma_{ess}(L)$ of $L$ on a non-trivial dense space $\cC \subset L^1(m)$ is reduced to $\{0\}$ we say that $L$ has \textbf{full Ruelle-Pollicott resonances} on $\cB$. This notion of \textbf{full Ruelle-Pollicott resonances} can be interpreted in terms of decorrelation (see Lemma \ref{lem:equivalence} for a link between the decay of decorrelation and the spectrum of the transfer operator $L$) as the following \textbf{decorrelation property}.

\begin{defin}[Deterministic decorrelation property]
    Let $L\in L(\cB)$ be an operator, it is said to satisfy \textbf{decorrelation property} if there are $M\in\bN$, \(\{\lambda_i\}_{i=1}^M, \lambda_i\in\bC\), an associated \(m_i\in \bN\), such that for any $i\leq M$ and any $k\leq m_i$, there are bilinear maps $C'_{i,k} : \cB\times \cB\mapsto \mathbb{C}$ such that

\begin{align}\label{eq:decoform}
    \int_I L^n\phi \varphi dm =\ \sum\limits_{i=1}^M\sum\limits_{k=0}^{m_i-1}\lambda_i^n n^k C'_{i,k}(\phi,\varphi) + o(\ve^n).
\end{align}
\end{defin}
In the literature, there is extensive study of the deterministic dynamical systems that realise full Ruelle-Pollicott resonances for example uniformly expanding maps \cite{Baladi2000}, uniformly hyperbolic maps \cite{GL2006, Baladi16}, Axiom A systems \cite{Ruelle86}, hyperbolic Blaschke products \cite{SBJ17} and more.

It is the latter definition that we will adapt to the families of i.i.d.\ random dynamical systems that we present in the next subsection. 

\begin{remark}
\begin{itemize}
    \item  Whenever the dynamical system $(X,T,m)$ is exponentially mixing with an invariant measure $\mu \ll m$,\footnote{\(\mu\ll m\) is a notation meaning that \(\mu\) is absolutely continuous wrt \(m\).} the leading eigenvalue $\lambda_1$ of $L$ is $1$, the multiplicity is $m_1=1$ and the corresponding bilinear form is given by  $C'_{1,0}(\phi,\varphi)=\int \phi d\mu \int \varphi dm$.
\item Notice that the Ruelle-Pollicott spectrum given by the terms $(\lambda_i)_{i\in\bN}$ may vary with the choice of the space $\cC$. A good example of that is the case of a Markov piecewise linear map of the interval $I=[0,1]$. When the observables are periodic functions in $C^\infty(I)\simeq C^\infty(\mathbb S^1)$, the Ruelle-Pollicott spectrum is reduced to $\{0,1\}$ (see \cite{Baladi2000}), but on $C^\infty([0,1])$, the Ruelle-Pollicott spectrum gets more interesting (see \cite{BKL22} ).
\end{itemize}
   \end{remark}

\subsection{Resonances in i.i.d.\ random dynamical systems.}\label{subseciidintro}

Let $(X,\cA,m)$ be a metric measure space (\(\cA\) is the Borel sigma-algebra on \(X\)) and $(\Omega,\cF,\mathbb P)$ a probability space. Let $(T_\omega)_{\omega\in\Omega}$ be a measurable family of measurable maps $T_\omega:X\to X$, meaning that the map
\[
(\omega,x)\mapsto T_\omega(x)
\]
from $\Omega\times X$ to $X$ is $(\mathcal F\otimes\mathcal \cA,\mathcal \cA)$-measurable.

As defined in previous subsection, for each $\omega\in\Omega$, let $U_\omega$ denote the Koopman operator
\[
U_\omega f := f\circ T_\omega ,
\]
and let $L_\omega$ be the corresponding transfer operator, defined by the duality relation
\[
\int_X L_\omega f\, g\, dm
=
\int_X f\, (g\circ T_\omega)\, dm,
\qquad
f\in L^1(m), \quad g\in L^\infty(m).
\]
Formally the annealed transfer operator $\cL$ on $L^1(m)$ corresponds to the following integral
\begin{equation}\label{eqoptrans}
\mathcal L f
:=
\int_\Omega L_\omega f\, d\mathbb P(\omega),
\end{equation}
whenever it is defined. 

Likewise, the annealed Koopman operator is defined by
\begin{equation}\label{eqopkoop}
\mathcal U f
:=
\int_\Omega U_\omega f\, d\mathbb P(\omega),
\end{equation}
whenever the integral is well defined.

 To introduce those two quantities a bit more rigorously on the Banach spaces we are interested in, let assume that $\cB$ is a reflexive Banach space (see Proposition \ref{prop:integomeg} in Section \ref{sec:classic} for $\cB$ a general Banach space\footnote{In order to define the annealed transfer operator as in \eqref{eqoptrans} it is enough to pick $\cB:=L^\infty(m)$ as soon as $U_\omega\in L^1(m)$ for all $\omega\in \Omega$. However if one wants to use quasi-compactness properties, then the notion of annealed transfer operator has to fit to stronger Banach spaces.}) and that the map $\omega \mapsto L_\omega\in L(\cB)$ (resp. $\omega \mapsto U_\omega\in L(\cB)$) is measurable (with respect to the Borel \(\sigma-\)algebra of $(L(\cB),\norm{\cdot}_{L(\cB)})$) and satisfies
\begin{align*}
    \int \norm{L_\omega }_{L(\cB)}d\bP<\infty.
\end{align*}
Then one can use the Bochner integral : for any measurable $A\in \mathcal{F}$, we can define for $f\in \cB$, $\int_{A} L_\omega fd\bP(\omega)$ as the unique vector in $\cB$ such that for any $g\in \cB^*$ (the dual of \(\cB\));

\[
\left\langle \int_A L_\omega fd\bP(\omega),g \right\rangle_{\cB,\cB^*} =\int_A \langle  L_\omega f,g \rangle_{\cB,\cB^*} d\bP(\omega).
\]

One can easily check that $g \mapsto \int_A \langle  L_\omega f,g \rangle d\bP(\omega)$ is a linear form on $\cB^*$ bounded from above by $\int\norm{L_\omega}_{L(\cB)}\norm{f}_{\cB}d\bP(\omega)$.
Since $\cB$ is reflexive, it defines an operator $(f\mapsto \int_A L_\omega fd\bP)\in L(\cB)$ .  
Notice that as a result we also have 

\begin{align}
\norm{\int_A L_\omega d\bP}\leq \int_A \norm {L_\omega} d\bP.    
\end{align}

Now consider the following i.i.d.\ Probability space $(\Omega^{\bN_0},\cF^{\otimes \bN_0},\bP^{\otimes \bN_0})$, we introduce $\tau : (\omega_n)_{n\in \bN_0} \in \Omega^{\bN_0} \mapsto (\omega_{n+1})_{n\in \bN_0}\in \Omega^{\bN_0}$ the probability preserving transformation, and we associate to this probability space the random family $(T_{\bar\omega})_{\bar\omega\in \Omega^\bN_0}$ given by $T_{\bar\omega}:=T_{\omega}$ for any $\bar \omega \in \Omega^{\bN_0}$ such that $\bar\omega_0=\omega$. Such random dynamical system $(T_{\bar \omega})_{\bar \omega\in \Omega^{\bN_0}}$ on $(\Omega^{\bN_0},\cF^{\otimes \bN_0},\bP^{\otimes \bN_0})$ can be identified to a dynamical skew product $(\Omega^\bN_0\times X,F)$ where $F: \Omega^\bN_0\times X\mapsto \Omega^\bN_0\times X$ is given by
\begin{align*}
    F(\bar \omega,x):=(\tau(\bar \omega),T_{\bar \omega}x).
\end{align*}

Then on the space of observables $\phi : \Omega \times X\mapsto \mathbb{C}$ independent of the coordinate in $\Omega^{\otimes \bN_0}$, the transfer operator associated to $F$ with respect to  the measure $m\otimes \bP^{\otimes \bN_0}$ is actually 
\begin{align}
    \cL f&:=\int_\Omega L_{\omega} f d\bP(\omega).
\end{align}
Indeed for $\phi \in L^\infty(\Omega^{\otimes \bN_0}\times X)$ and  $\psi\in L^1(\Omega^{\otimes \bN_0}\times X)$ both independent of the coordinate in $\Omega^{\otimes \bN_0}$, that is, they can be written as $\phi(\bar \omega,x):=\bar\phi(x)$ and $\psi(\bar \omega,x):=\bar\psi(x)$ for any $\bar \omega\in \Omega^{\otimes \bN_0}$,
\begin{align*}
    \int_X\int_{\Omega^{\otimes \bN_0}} \phi(\bar \omega,x) \psi\circ F(\bar \omega,x)d\bP^{\otimes \bN_0}(\bar \omega)dm(x)&=\int_X\int_{\Omega^{\otimes \bN_0}}\int_\Omega \bar\phi(x) \psi(\tau(\bar \omega),T_{\omega_0}x)d\bP(\omega_0)d\bP^{\otimes \bN_0}(\tau \bar \omega)dm(x)\\
    &=\int_{\Omega^{\otimes \bN_0}}\int_X \int_\Omega L_{\omega_0}\bar \phi(x) \bar\psi(x)d\bP(\omega_0)dm(x)d\bP^{\otimes \bN_0}(\tau \bar \omega)\\
    &=\int_{\Omega^{\otimes \bN_0}}\int_X \cL\phi(\bar \omega,x) \psi(\bar \omega,x)dm(x)d\bP^{\otimes \bN_0}(\bar \omega).
\end{align*}

Thus the notion of Ruelle-Pollicott resonances introduced in Subsection~\ref{subsecdeter} for deterministic dynamical systems raises naturally to random dynamical systems in i.i.d.\ settings through the deterministic representation by a skew product and can be interpreted as the eigenspace associated to the isolated eigenvalues of $\cL$.
The following definition, that summarizes the notion of \textbf{full Ruelle-Pollicott resonances} for i.i.d.\ systems presented above, will be the focus of the paper and the ensuing Theorems.

The goal of this paper and our main results is to describe and locate the resonances for the random family of maps $(T_\omega)_{\omega\in \Omega}$ whose annealed operator is given by $\cL$.
\begin{defin}[Annealed decorrelation property]\label{def:main}
   We say that the annealed random dynamical system $(T_\omega)_{\omega\in \Omega}$ on $(\Omega,\cF,\bP)$ yields a \textbf{full Ruelle-Pollicott resonances} (also called \textbf{decorrelation formula}) for the measure $m$ and the dense space $\cC\subset L^1(m)$ if the following occurs. There exists a set of complex numbers \(\Xi_1 = \{\lambda_1, \lambda_2,\ldots\}\) and, for each \(\lambda_i\in\Xi\), an associated \(m_i\in\bN\) such that, for any \(\phi,\varphi\in \cC^\infty(\cI)\) and \(\ve>0\), there is an asymptotic expansion
\[
\int_I\phi\cdot \cU^n \varphi dm =\int_I\cL^n\phi\cdot  \varphi dm=\ \sum\limits_{\lambda_i\in\Xi:|\lambda_i|\geq \ve}\sum\limits_{k=0}^{m_i-1}\lambda_i^n n^k C'_{i,k}(\phi,\varphi) + o(\ve)
\]
where \(C'_{i,k}(\phi,\varphi)\) are finite rank and non-zero bilinear functions of \(\phi\) and \(\varphi\).
\end{defin}

The Ruelle-Pollicott resonances enjoy properties analogous to those arising in deterministic dynamical systems. For example, suppose that there exists an ergodic stationary probability measure \(\mu\ll m\), that is, \(\mu\otimes\mathbb P\) is invariant and ergodic under the skew-product transformation \(F\). If, moreover, \(\mu\) admits a density belonging to \(\cB\) and the annealed dynamics is mixing with respect to \(\mu\), then the leading eigenvalue of \(\cL\) is 1, it is simple, and the corresponding contribution to the decorrelation formula is given by
\[
C'_{1,0}(\phi,\varphi)
=\Bigl(\int \phi d\mu\Bigr)
\Bigl(\int \varphi dm\Bigr).
\]
which, thus, is also the leading term in the asymptotic expansion while the remaining resonances describe the decay of correlations towards this limit.

The main objective of this paper is to bridge the spectral properties of the annealed transfer operator \(\cL\) to those of the fibre operators \((L_\omega)_{\omega\in\Omega}\). In the deterministic setting, two important classes have been extensively studied: compact transfer operators \cite{Pollicott86,SBJ17} and quasi-compact transfer operators \cite{Baladi2000,GL2006,BKL22}. We develop analogous frameworks for random families of maps and investigate how the corresponding Ruelle-Pollicott spectrum are determined by the spectral properties of the fibres.

We begin with the case in which the family \((T_\omega)\) is such that the associated transfer operator \(L_\omega\in L(\mathcal B)\) is compact for every \(\omega\in\Omega\).

\begin{proposition}\label{prop:gencomp}
Assume that there is a Banach space $\cB$ such that
$\omega\mapsto L_\omega\in L(\cB)$ is a measurable family of compact operators. If $\int \norm{L_\omega}_{L(\cB)}d\bP(\omega)<\infty$, then the annealed operator 
$$
 \cL:=\int L_\omega(\cdot)d\bP(\omega)
 $$
 is a compact operator on $\cB$ and thus satisfies the \textbf{decorrelation property}.
\end{proposition}

\begin{proof}
   The proof is done in appendix, see Section \ref{sec:classic}.
\end{proof}

Compact transfer operators arise naturally in several classes of random dynamical systems. We study two different cases of compact annealed operators namely the one associated to stochastic differential equations on a \(d\)-dimensional torus see Section~\ref{sec:sde} and for an annealed operator associated to a family of hyperbolic Blaschke products defined on the 2-torus \(\bT^2\), see Section~\ref{secblaschke}.

In the deterministic case of Blaschke products, one can accurately locate the Ruelle-Pollicott resonances for this type of families \cite{SBJ17}. Blaschke products are an interesting family because they are analytic hyperbolic diffeomorphisms of the torus, and a special case where the resonances of the annealed operator are actually the average of the resonances of the deterministic system. This is the case because the general eigen space for all the operators turns out to be the same.

\begin{mainthm}\label{thm:inverse}
Given \(\omega\in\bC\) with \(|\omega|<1\), any family of Blaschke products 
        $T_\omega :\bT^2\mapsto \bT^2$ defined by
\[
    T_{\omega}(z_1,z_2)=\left(z_1\frac{z_1-\omega}{1-\bar{\omega}z_1}z_2,\frac{z_1-\omega}{1-\bar{\omega}z_1}z_2\right)
\]

such that the associated annealed Koopman operator (whenever well defined)
\[
\cU:=\int U_\omega d\bP(\omega),
\]
is compact. Furthermore its spectrum is given by
 \[
 \sigma(\cU) = \left\{\int_\Omega (-\omega)^n : n\in\bN\right\} \cup \left\{\int_\Omega(-\bar\omega)^n : n\in\bN\right\}\cup \{0,1\},
 \]
 with the multiplicity of $\lambda\in \sigma(\cU)\backslash\{0,1\}$ given by $m_\lambda:=\#\{n, \int_\Omega (-\omega)^n=\lambda\}+\#\{ n, \int_\Omega(-\bar\omega)^n=\lambda\}$ and\footnote{ \(\#A\) denotes the cardinality of the set \(A\).} $1$ being a simple eigenvalue.
\end{mainthm}

Theorem \ref{thm:inverse} above is given for the Koopman operator but this spectral data also holds for the Transfer operator $\cL$ with respect to the Volume measure $m$ as a reference measure. It follows that such random dynamical system admits \textbf{full Ruelle-Pollicott resonances}. And since $1$ is a simple eigenvalue and the only one on the unit circle, $\cL$ is also mixing with the volume measure $m$ being the stationary measure.\\

Since Ruelle-Pollicott resonances correspond to isolated eigenvalues of the spectrum of the transfer operator, the operators exhibiting non-trivial Ruelle-Pollicott resonances are necessarily quasi-compact. It is therefore natural to consider i.i.d.\ families for which the fiberwise transfer operators $L_\omega$ are quasi-compact.
This setting is more delicate when one aims to establish the existence of \textbf{full Ruelle-Pollicott resonances} for random dynamical systems. Indeed, quasi-compact operators typically possess a non-trivial essential spectrum on any given Banach space. As a consequence, a spectral analysis on a single Banach space is insufficient to capture the full resonance structure.
To overcome this difficulty, we consider the spectrum over a suitable sequence (or lattice) of Banach spaces. This approach allows for a refined spectral decomposition and serves as guidelines summarized in the two statements within the following theorem.

\begin{proposition}\label{propgenquasi}
Let $(W_r)_{r\in \bN}$ be a family of Banach spaces.
\begin{enumerate}
    \item\label{item1:propgenquasi} Let $(L_\omega)_{\omega}$ be a family of transfer operators such that $L_\omega : W_r\mapsto W_r$ is well defined for any $r\in \bN$ and $\omega \mapsto L_\omega$ is measurable in $L(W_r)$. Assume that
\begin{align}\label{eqthmegenquasizero}
    L_\omega=\Pi_{\omega,r}+N_{\omega,r},
\end{align}
With $\Pi_{\omega,r}$ a compact operator and $N_{\omega,r}$ an orthogonal operator such that there is $C>0$ and that for a.e $\omega \in \Omega$, $\norm{ N_{\omega,r}}_{W_r}\leq \ve_r(\omega)< 1$ with $\ve_r(\omega)\underset{r\to \infty}{\to}{0}$.
Let $\eta>0$, then for $r$ large enough, $\cL:=\int L_\omega d\bP(\omega)$ is quasi compact and can be decomposed as an operator on $W_r$ into
\begin{align}\label{eqthmegenquasi}
    \cL=\Pi_{r}+N_{r},
\end{align}
with $\Pi_r$ a compact operator orthogonal to the operator $N_r$ and with the latter satisfying
\begin{align}
    \norm{N_r}_{W_r}\leq \int \ve_r(\omega)d\bP(\omega)<\eta.
\end{align}

\item\label{item2:propgenquasi} Let $\cL :W_1\mapsto W_1$ a continuous operator such that $\cL_{|W_r}:W_r\mapsto W_r$ is continuous and admits a spectral gap with shrinking essential spectrum  : for $r\in \mathbb{N}$,  there is $C_r>0$ and $\ve_r<1$ such that 
\begin{align*}
    \cL_{|W_r}=\Pi_r+N_r,
\end{align*}
with $\Pi_r$ a compact operator orthogonal to the operator $N_r$ and the latter satisfying $\|N_r^n\|\leq C_r \ve_r^n$. Furthermore,
assume that for any $r>0$, $(W_r,\|\cdot\|_r)\subset (W_{r-1},\|\cdot\|_{r-1})$ with continuous injection and $\bigcap_{r>0}W_r$ is dense inside $(W_r,\norm{\cdot}_{r})$ and for any eigenvalue $\lambda_i$ of $\cL$ on $W_r$ with $|\lambda_i|>\ve_r$, the associated general eigenspace\footnote{that is the space $E_r(\lambda):=\oplus_{n\in \mathbb{N}}Ker(\lambda I-\cL_{|W_r})$} $E_r(\lambda_i)$ belongs to $\bigcap_{r\in \bN} W_r$ , then for any $f,g\in \bigcap_{r>0}W_r$
\begin{align}
\int \cL^nf \cdot g d\mu = \sum\limits_{|\lambda_i|\geq\ve}\sum\limits_{j< m_i}\lambda_i^n n^j c_{i,j}(f,g) + o(\ve^n),    
\end{align}
where $m_i$ is the geometric multiplicity of the eigenvalue $\lambda_i$ for $\cL$ on $W_r$ for any $r>0$ such that $\int \ve_r(\omega)d\bP(\omega)\leq \ve$, and $c_{i,j}:\bigcap_{r\in \bN} W_r\times \bigcap_{r\in \bN} W_r\mapsto \bC$ a bilinear map independent of the choice of $r$.
\end{enumerate}

\end{proposition}

\begin{proof}
    See section \ref{sec:classic} for a proof of the Proposition.
\end{proof}

Among families $(T_\omega)_{\omega\in \Omega})$ whose fiberwise transfer operators $L_\omega$ are quasi-compact, we focus in Section~\ref{sec:piecexpand} on the class of Markov transitive piecewise linear expanding maps of the unit interval. Such families encompass many families of maps including the subclass of $\beta$-transformations for $\beta \in \mathbb{N}$. In Section~\ref{sec:piecexpand}, we prove the existence of \textbf{full Ruelle-Pollicott resonances} for such families and investigate techniques to localize the spectrum of the associated annealed operator. The main statement of this Section can be summarized into the following Theorem \ref{mainthm:B} (see Theorem \ref{thm:proj} for a detail version).

\begin{mainthm}\label{mainthm:B}
 (see Theorem \ref{thm:proj} for a general statement) Let $\tilde \cI$ be a partition of the interval $I$ and 
let $\omega \in \Omega \mapsto T_\omega$ be a measurable\footnote{the family is measurable for the $\sigma-$algebra of Borelians on $(C^\infty(\tilde \cI),\norm{\cdot}_\infty)$} family of transitive piecewise linear expanding Markov maps with Markov partition $\cI_\omega$ finer than $\cI$ and with $T_\omega(I_\omega)\cap I\in \{\emptyset, I\}$. Assume there is $1<\lambda$ such that for any $\omega\in \Omega$ and any $x\in I$, $\lambda\leq |T_\omega'(x)|$ and that $ \Lambda:=\int_\Omega \norm{T_\omega'}_\infty d\bP(\omega)<\infty $,
then there is a family $(W_r)_{r\in \bN}$ of Banach spaces  such that for $r\in \bN$,  $\omega \mapsto L_\omega \in (L(W_r),\norm{\cdot}_{W_r\mapsto W_r})$ is measurable
and $\cL=\int L_{\omega} d\bP(\omega)$ satisfies the \textbf{decorrelation property}.
\end{mainthm}

As a byproduct of the proof of Theorem \ref{thm:proj}, the annealed operator $\cL$ is mixing, that is $\lambda=1$ is the only eigenvalue on the unit disc and we have a unique stationary measure $\mu$ absolutely continuous with respect to the Lebesgue measure. Notice however that, as expected but contrary to the Blaschke product case, the fiber wise operators may not share the same invariant measure $\mu$.

\section{Annealed operator of families of piecewise-linear expanding maps}\label{sec:piecexpand}

 In this section we investigate random families of piecewise linear expanding Markov interval maps of \(I=[0,1]\). The ergodic properties of such random families have been studied in \cite{ANV2015} but the description of the Ruelle-Pollicott spectrum has only been given in the deterministic case, that is for a given piecewise linear expanding interval map $T$ of \(I=[0,1]\) one map. The \textbf{full Ruelle-Pollicott resonances} property can be found in \cite{BKL22}. The challenge to check \textbf{full Ruelle-Pollicott resonances} lay on the fact that families of maps can have different Markov partitions and a key element to apply the framework of Proposition \ref{propgenquasi} is to find a family of Banach spaces well adapted to any of these Markov maps. This will be done through the introduction of the sets $C^\infty(\cI,\tilde \cI)$ and $W_r(\cI)$ that emphasize a distinction between the Markov partition $\cI$ of a map $T$ and the partition $\tilde \cI$ generated by $\{T(I),I\in \cI\}$. The latter one is what help us constructing the desired Banach spaces. We introduce in what follows definitions and relevant quantities to study Markov linear expansions and the key Banach spaces. Then we will prove full Ruelle-Pollicott resonances using the properties from \cite{BKL22} and the framework of Proposition \ref{propgenquasi}.
 
Let \(I=[0,1]\) and let $(\Omega,\bT,\bP)$ be a probability space to which we adjoin a family of piecewise affine expanding Markov maps $F=\{(T_\omega)_{\omega \in \Omega}\}$ such that $(x,\omega)\mapsto T_\omega(x)$ is $I\otimes B(I)$ measurable, and the family $F$ satisfies the following :
 There is a quasi-partition $\tilde \cI$ of $I$ made of a collection of disjoint open intervals $\{\tilde I_i\}_{i=1}^N = \{(a_i, a_{i+1})\}_{i=1}^N$ such that for any $\omega\in \Omega$, each $T_\omega \in F$, admits a finite partition  $\cI_\omega$ made of disjoint intervals such that any element of $\tilde \cI$ is made of elements of $\cI_\omega$ and\footnote{In other words, \(\tilde \cI\) is coarser than \(\cI_\omega\) for every \(\omega\in\Omega\).} $T_\omega'$ is constant on it. Furthermore, we assume that for $I_\omega \in \cI_\omega$ and $\tilde I \in \tilde \cI$,

\[
\text{either}\quad T_\omega(I_\omega)\cap \tilde I=\emptyset \quad \text{or}\quad \tilde I\subset T_\omega(I_\omega).
\]

In particular when, for all $\omega\in \Omega$, $\tilde \cI=\cI_\omega$, then the maps $(T_\omega)_{\omega\in \Omega}$ are Markov maps with a common Markov partition $\tilde \cI$. For each $\omega\in \Omega$, we choose to enumerate the $M(\omega)$ elements of $\cI_\omega$ and for $I_{\omega,i}\in \cI_\omega$, we will denote by $\lambda_{\omega,i}:=T_\omega'|_{I_{\omega,i}}$.

We use \(\cC^\infty(\cI,\tilde \cI)\)  to denote the set of functions $T$ on \(I\) which are \(\cC^\infty\) when restricted to each \(I_j\in  \cI\) and whose image $T(I_j)$ is a union of element of  \(\tilde \cI\). We will denote \
$$
\cC^\infty(\tilde \cI):=\bigcup_{\tilde \cI\preceq \cI}\cC^\infty(\cI,\tilde \cI),$$
 where $\tilde \cI\preceq \cI$ means that the elements of $\tilde \cI$ can be written as union of elements of $\cI$.  
Finally, we use \(\sigma_{\cB}(\cL)\) for the spectrum of an operator \(\cL\) acting on a Banach space \(\cB\).

Observe that for any \(r\geq 0\), the Sobolev space \(W^{r,1}(I)\) is the set of all \(h \in L^1(I)\) such that \(h\) and all of its weak derivatives up to the \(r\)th belong to \(L^1(I)\). We define the space \(W^{r,1}(\tilde \cI)\) as the set of all \(h \in L^1(I)\) such that, for each \(\tilde I_i\in \tilde \cI\), \(h|_{\tilde I_i}\in W^{r,1}(\tilde I_i)\). For convenience we write \(h'\) and \(h^{(l)}\) to mean the weak derivative and the \(l-\)th weak derivative respectively of \(h\) restricted to elements of \(\tilde \cI\). For each \(r \in \bN_0\) the space \(W^{r,1}(\tilde \cI)\) is a Banach space equipped with the norm
\[
\norm{h}_{r,1} = \sum\limits_{l=0}^r \int_\cI |h^{(l)}(x)|\  dx. 
\]

In the following, to simplify notation, we will write \(W_r\) for \(W^{r,1}(\tilde \cI)\) and we will write \(\norm{\cdot}_r\) for \(\norm{\cdot}_{r,1}\). Observe that \(W_0\) coincides with \(L^1(I)\).

For \(\omega\in\Omega\), define a \(M(\omega)\times N\) matrix \(A_\omega\) as 
\[
A_\omega[i,j] = 1,\ \text{if}\ \tilde I_j\subset T_\omega(I_{\omega,i})\quad
\text{and } 
A_\omega[i,j]=0,\ \text{if}\ T_\omega(I_{\omega,i})\cap \tilde I_j=\emptyset,
\]
which is the \emph{adjacency matrix} of the map \(T_\omega\). For \(k\in \bN_0\) and \(\omega\in\Omega\), let \(B_{\omega,k}\) be the \(N\times M(\omega)\) matrix defined by 
\begin{align}\label{eq:bkformula}
    B_{\omega,k}[i,j] := \lambda_{\omega,j}^{-k} A_\omega[j,i],
\end{align}

where (we recall that) $\lambda_{\omega,j}:=T_\omega'|_{I_{\omega,j}}$.

For such a map we define the Transfer operator $L_\omega$, associated to \(T_\omega\) as follows : for $x\in I$,

\begin{equation}\label{eq:transfer_operator}
L_{\omega,k} h(x):=\sum_{y\in T_\omega^{-1}(x)}\frac{h(y)}{|T_\omega'(y)|^k}=\sum_{1\leq i \leq N} 1_{\tilde I_i}\sum_{1\leq j \leq M(\omega)}B_{\omega,k} [i,j]h\circ g_{\omega,j}     
\end{equation}

with $B_{\omega,k} [i,j]=\lambda_{\omega,j}^{-k} A_\omega[j,i]$ where $A_\omega[i,j]=1$ if $\tilde I_j\subset T_\omega(I_{\omega,i})$ and $0$ otherwise, and \(g_{\omega,j}=T_\omega|_{I_{\omega,j}}^{-1}\).

Since 
composition with an affine transformation preserves Sobolev space and the sum consists of a finite number of terms it follows that these operators are
well defined as operators \(L_{\omega,k} : W_r(\tilde\cI) \to W_r(\tilde \cI)\). Similarly they are well defined, by this same formula, on \(\cC^r(\cI) =\oplus_i \cC^r(I_i)\).\\
Notice that since $ \tilde \cI \preceq \cI_\omega$,
$T_\omega$ is a Markov operator on $\cI_\omega$. However, instead of looking at $L_{\omega,k}$ as an operator on $W_r(\cI_\omega)$ as in \cite{BKL22}, we look at it as an operator on $W_r(\tilde \cI)$. The two Banach spaces can be linked through the following lemma :

\begin{lemma}\label{lemextpartition}
If $ \tilde \cI \preceq \cI_\omega$, then any $h\in W_r(\tilde\cI)$ satisfies $h\in  W_r(\cI_\omega)$ and $\norm{h}_{W_r(\tilde \cI)}=\norm{h}_{W_r(\cI_\omega)}$.
\end{lemma}

\begin{proof}
Recall that any interval $I_i\in \tilde \cI$ is a union of elements $I_{\omega,1},\dots,I_{\omega,k_i}\in \cI_\omega$. Thus for any $h\in W_r(\tilde \cI)$, 
\begin{align*}
    \norm{h}_{W_r(\tilde \cI)}&=\sum_{l\leq r}\sum_{I_i\in \tilde \cI}\int |h^{(l)}|1_{I_i}dx\\
    &=\sum_{l\leq r}\sum_{I_i\in \tilde \cI}\int \sum_{j=1}^{k_i}|h^{(l)}|1_{I_{\omega,j}}dx\\
    &=\norm{h}_{W_r(\cI_\omega)}.
\end{align*}
Thus $h\in W_r(\cI_\omega)$ and the Lemma is proven.
\end{proof}

We define the space \(\fB_r(\cI)\) of piecewise polynomial functions of degree at most \(r\) on each interval $I_i\in \cI$.

We now state the main Theorem of this section here :

\begin{theorem}\label{thm:proj}
Let $\tilde \cI$ be a partition of the interval $I$ and
let $\omega \in \Omega \mapsto T_\omega\in (C^\infty(\tilde \cI),\norm{\cdot}_\infty)$ be a measurable family of transitive piecewise linear expanding maps for the Borel $\sigma-$algebra of $(C^\infty(\tilde \cI),\norm{\cdot}_\infty)$. Assume there is $1<\lambda$ such that for any $\omega\in \Omega$ and any $x\in I$, $\lambda\leq |T_\omega'(x)|$ and that $ \Lambda:=\int_\Omega \norm{T_\omega'}_\infty d\bP(\omega)<\infty $,
then for $r\geq 1$,  $\omega \mapsto L_\omega \in (L(W_r),\norm{\cdot}_{W_r\mapsto W_r})$ is measurable
and $\cL_k=\int L_{\omega,k} d\bP(\omega)$ satisfies the \textbf{decorrelation property}.
\end{theorem}

\begin{remark}
    As a by-product of the proof of Theorem \ref{thm:proj}, we will prove that the space  \(\fB_r(\tilde \cI)\) of piecewise polynomial functions of degree \(r\) on each interval $I_i\in \tilde \cI$ is a finite dimensional invariant space for the annealed operator $\cL_k$. This will be an important feature to read the resonances on the Jordan form of the matrix representation of $\cL_k$ on $\fB_r(\tilde \cI)$ we will expand on it in the current section after the proof of Theorem \ref{thm:proj}.
\end{remark}

The strategy of the proof consists of using the framework of Proposition \ref{propgenquasi} to obtain the decorelation property \eqref{eq:decoform}. We proceed in two steps, \textbf{Step 1} consists of proving that the family $\omega\mapsto T_\omega$ generates a measurable family $\omega\mapsto L_{\omega,k}\in L(W_r)$ for the right $\sigma$-algebra. In \textbf{Step 2} we put ourselves in the conditions for the second statement of Proposition \ref{propgenquasi} to hold. That is we prove that $\cL_k$, as an operator over $W_r(\tilde \cI)$, has a shrinking essential spectrum while $r$ is growing. 
Then we prove that the eigenspaces of the annealed operator $\cL_k$ associated to large enough eigenvalues are embedded in the space $\cC^\infty(\tilde{\cI},\tilde{\cI})\subset \bigcap_{r\in \mathbb{N}} W_r$ thus checking the second statement of Proposition \ref{propgenquasi}.\\

\textbf{Step 1 : measurability of $\omega \mapsto L_{\omega,k}$} 
In order to talk about the annealed operator, we give some justification for the measurability of the family $\omega \mapsto L_{\omega,k}$ which are dealt with through the technical Lemma \ref{lemmesurabl} and  the following Proposition \ref{prop:measop}.

\begin{proposition}\label{prop:measop}
Let $\omega \mapsto T_\omega \in C^\infty(\cI)$ be a measurable family, then for any $r>0$, $\omega \mapsto L_{\omega,k}(\cdot)\in L(W_r(\cI))$ is measurable. 
\end{proposition}

\begin{proof}
Fix $h\in C^r(\cI,\cI)$. We will first show the continuity of $T\in (C^\infty(\cI),\norm{\cdot}_\infty)\mapsto L_{T,k}h\in (W_r(\cI),\norm{\cdot}_{r(I)})$. We will then conclude using Lemma \ref{lem:equivalence}.

\textbf{Step 1, continuity of $T\in (C^\infty(\cI),\norm{\cdot}_\infty)\mapsto L_{T,k}h\in (L^1(\cI),\norm{\cdot}_{L^1(I)})$ for $h\in C^r(\cI)$:}\\
Set $T_0\in C^\infty(\cI)$, Without loss of generality one can associate to $T_0$ a Markov partition $\cI_0$ such that for any $J\in \cI_0$, $T_0(J)\in \cI$. Let $A:=\sup\{|H|,H\in \cI\}$ and\footnote{Here \(|H|\) denotes the diameter of any set \(H\).} fix $\ve\ll A$. Recall that from Formula \eqref{eq:transfer_operator}, the operator $L_{T_0,k}f$ can be expressed as
\begin{align*}
    L_{T_0,k} h(x)=\sum_{y\in T_0^{-1}(x)}\frac{h(y)}{|T_0'(y)|^k}=\sum_{\tilde I_i\in \cI} 1_{\tilde I_i}\sum_{I_j\in \cI_0}B_{0,k} [i,j]h\circ g_{0,j},
\end{align*}
where $g_{0,j}(x):=(T_0|_{I_j})^{-1}(x)=\lambda_{0,j}x+p_{0,j}$ for each $I_j\in \cI_0$. Now choose $T_1\in C^\infty(\cI)$ $\ve$-close to $T_0$ and with Markov partition $\cI_1$ (this Markov partition is again taken without loss of generality the finest possible). To prove that $L_{T_1,k}h$ is close to $L_{T_0,k} h$, it is enough to prove first, that the two partitions $\cI_0,\cI_1$ have the same number of elements, that all those elements can be paired into couples $(I_{0,j},I_{1,j})$ with $I_{0,j}\in \cI_0$ and $I_{1,j}\in \cI_1$ with each pairing being close for the Haussdorf distance $d(I_{0,j},I_{1,j}):=|I_{0,j}\Delta I_{1,j}|$. And second that the corresponding terms $\lambda_{i,j}$ and $p_{i,j}$ are also close.\\

Let $I_{0,j}\in \cI_0$ and $J\in \cI$ such that $T_0(I_{0,j})=J$. Let $x_0\in I_{0,j}$ be the middle point of $I_{0,j}$, Then by construction, $T_0(x_0)$ is the middle point of $J$ and since $\ve\ll |J|$, $x_0$ belongs to the interior of an interval that we index as $I_{1,j}$ such that $T_1(I_{1,j})=J$ (actually, because we took the Markov partition the finest possible, $x_0$ is also the middle point of that interval). We now prove that $d(I_{0,j},I_{1,j})$ shrinks with $\ve$. 

Let $\check x:=\inf I_{0,j}\wedge I_{1,j}$ and $\hat x:=\sup I_{0,j}\wedge I_{1,j}$. Notice that for all $i\in \{0,1\}$, $T_i(\check x)-T_i(\hat x)=\lambda_{i,j}(\check x-\hat x)$ and since $\norm{T_1-T_0}_\infty\leq \ve$,

\begin{align}
|T_1(\check x)-T_0(\hat x)|\leq \ve. 
\end{align}

Notice that for all $i\in \{0,1\}$ $|T_i(\check x)-\inf J|\leq \ve$ and $|T_i(\hat x)-\sup J|\leq \ve$.

\begin{align*}
|T_1(\hat x)-T_1(\check x)-T_0(\hat x)+T_0(\check x)|&\leq 2\ve \\
|\frac{T_0(\hat x)-T_0(\check x)}{T_1(\hat x)-T_1(\check x)}-1|&\leq \frac{2\ve}{T_1(\hat x)-T_1(\check x)} \\
|\frac{\lambda_{0,j}}{\lambda_{1,j}}-1|&\leq \frac{2\ve}{|J|-2\ve} \\
\end{align*}
Thus for $\ve$ small enough such that $|J|\gg \ve$ we already get that $\lambda_{1,j}$ must be close to $\lambda_{0,j}$ and we can prove that 
\begin{align*}
|I_{0,j}\Delta I_{1,j}|&=|I_{0,j}|+|I_{1,j}|-2(\hat x-\check x)\\
&\leq \frac{|J|}{\lambda_{0,j}}+\frac{|J|}{\lambda_{1,j}}-2\frac{|J|-2\ve}{\lambda_{0,j}}\\
&\leq \frac{|J|}{\lambda_{0,j}}\left(1+\frac{\lambda_{0,j}}{\lambda_{1,j}}-2\left(1-2\frac{\ve}{|J|}\right)\right)\\
&\leq  \frac{|J|}{\lambda_{0,j}}\left(2+\frac{2\ve}{|J|-2\ve}-2\left(1-2\frac{\ve}{|J|}\right)\right)\\
&\leq  \frac{|J|}{\lambda_{0,j}}\left(\frac{2\ve}{|J|-2\ve}+4\frac{\ve}{|J|}\right)\\
&\leq  \frac{|J|}{\lambda_{0,j}}\left(6\frac{\ve}{|J|-2\ve}\right)\\
&\leq  \frac{6\ve |I_{0,j}|}{|J-2\ve|}.
\end{align*}
Notice that each $I_{0,j}$ is matched to an element $I_{1,j}$. Now we argue that two consecutive elements $I_{0,j}$ and $I_{0,j+1}$ implies having consecutive $I_{1,j}$ and $I_{1,j+1}$
Indeed, the equation above implies that between two elements $I_{1,j}$ and $I_{1,j'}$ corresponding to  two consecutive elements $I_{0,j}$ and $I_{0,j+1}$ there is only room for an interval $I_1$ of size at most $\frac{4\ve}{A}$. Assume that there is such interval and denote by $J_1\in \tilde \cI$ the image $T_1(I_1)$. Let $x_j$ be the matching extremity between the intervals $I_{0,j}$ and $I_{0,j+1}$, then $x_j$ is in the interior of $I_1$ and thus  
\begin{align*}
|T_0(x_j^-)-T_0(x_j^+)|&=|T_0(x_j^-)-T_1(x_j^-)+T_1(x_j^+)-T_0(x_j^+)| \\
&\leq 2\ve.
\end{align*}
Since we have $A\gg\ve$, the inequality above is satisfied only if $T_0(x_j^-)=T_0(x_j^+)=T_0(x_j)$. But $T_1(x_j)$ is at distance at least $\frac{|J_1|}{2}$ of one of the extremity of $J_1$, one can assume without loss of generality that there is $x_1\in I_1$ with $|T_1(x_1)-T_1(x_j)|\geq \frac{|J_1|}{2}$ and $x_1\in I_{0,j}$, then recall that since $\lambda_{0,j}=\frac{|J|}{|I_{0,j}|}$,
\begin{align*}
    |T_0(x_j)-T_0(x_1)|\leq \frac{4\ve (|I_{0,j}|+|I_{1,j}|)}{|J|}\frac{|J|}{|I_{0,j}|}&\leq 4\ve (1+\frac{|I_{0,j}|+|I_{1,j}|}{|J|})\\
    &\leq  4\ve (1+\frac{2}{|J|}).
\end{align*}
Since $|J|\gg \ve$, we deduce by chain rule that
\begin{align*}
  |T_0(x_j)-T_1(x_j)|&= |T_0(x_j)-T_0(x_1)+T_0(x_1)-T_1(x_1)+T_1(x_1)-T_1(x_j)|\\
  &\geq \frac{|J_1|}{2}-\ve-4\ve (1+\frac{2}{|J|})\\
  &\gg \ve.
\end{align*}
This yields a contradiction with the fact that $\norm{T_0-T_1}_\infty\leq \ve$. Thus the sets $I_{1,j}$ enumerate the partition $\cI_1$. 

Finally, $p_{0,j}$ and $p_{1,j}$ are close : since $\norm{T_1-T_0}_\infty \leq \ve$,
\begin{align*}
    |T_1(\check x)-T_0(\check x)|\leq \ve\\
    |(\lambda_{0,j}-\lambda_{1,j})\check x +p_{1,j}-p_{0,j}|\leq \ve\\
    |(1-\frac{\lambda_{1,j}}{\lambda_{0,j}})\check x +\frac{p_{1,j}-p_{0,j}}{\lambda_{0,j}}|\leq \frac{\ve}{\lambda_{0,j}}\\
    |\frac{p_{1,j}-p_{0,j}}{\lambda_{0,j}}|-|\frac{\ve}{|J|}| \leq \frac{\ve}{\lambda_{0,j}}\\
    |p_{1,j}-p_{0,j}|-|\frac{\ve \lambda_{0,j}}{|J|}| \leq \ve\\
     |p_{1,j}-p_{0,j}|-|\frac{\ve }{|I_0|}|\leq \ve
\end{align*}
thus
\begin{align*}
     |p_{1,j}-p_{0,j}|\leq \frac{\ve }{|I_0|}+ \ve.
\end{align*}

Now collecting everything we can safely argue that for $\delta >0$, there is $\ve>0$ small enough such that for any index $i,j$, $B_{0,k}[i,j]=B_{1,k}[i,j]$ and 
\begin{align*}
   \int \left|L_{k,T_1}f-L_{k,T_0}f\right| dx &\leq \int  \sum_{\tilde I_i\in \cI} 1_{\tilde I_i}\sum_{I_j\in \cI_0}B_{0,k}[i,j]|h\circ g_{0,j}-h\circ g_{1,j}|dx\\
   &\leq   \sum_{\tilde I_i\in \cI}\sum_{I_j\in \cI_0}B_{0,k}[i,j] \int_{\tilde I_i}|h\circ g_{0,j}-h\circ g_{1,j}|dx\\
\end{align*}
Since $p_{i,j}$ and $\lambda_{i,j}$ are close to one another, the map $g_{0,j}$ and $g_{1,j}$ remain close, and since $h\in C^r\subset C^0$ the integral above remains close to $0$. So we proved the continuity of 
$T\in (C^\infty(\cI),\norm{\cdot}_\infty)\mapsto L_{T,k}h\in (L^1(\cI),\norm{\cdot}_{L^1})$.\\

\textbf{Step 2,  $T\mapsto L_{T,k}h\in (W_r,\norm{\cdot}_{r})$ is continuous for any $h\in W_{r+1}$:}\\
Now recall from Lemma \ref{lemlomega} that for $k\in \bN_0$ and $l\leq r$, 
\begin{align*}
(L_{T,k}h)^{(l)}=L_{T,k+l}(h^{(l)}),
\end{align*}
and since $h^{(l)}\in C^0$, the continuity of  $T\mapsto (L_{T,k}h)^{(l)}\in (L^1,\norm{\cdot}_{L^1})$ follows from Step 1.
Thus $T\mapsto L_{T,k}h\in (W_r,\norm{\cdot}_{r})$ is continuous.\\

\textbf{Step 3, $\omega \mapsto L_{\omega,k}\in L(W_r,\norm{\cdot}_{r})$ is measurable :}\\
According to Lemma \ref{lemmesurabl}, since $W_r$ is separable and one can make a dense family from elements of $C^r(\cI)$, one has that $\omega \mapsto L_{\omega,k}\in (L(W_r),\norm{\cdot}_{L(W_r)})$ is measurable.
\end{proof}

\textbf{Step 2: Spectral decomposition for the annealed operator $\cL_k$ :}

Having a spectral gap for $L_{\omega,k}$ on $W_r$ is not enough to get item~(\ref{item1:propgenquasi}) from Proposition \ref{propgenquasi}, since item~(\ref{item1:propgenquasi}) requires that the operator $N_{\omega,k}$ encompassing the essential spectrum in equation \eqref{eqthmegenquasizero} of Proposition~\ref{propgenquasi} to be a contraction. So, instead of using item~(\ref{item1:propgenquasi}) from Proposition~\ref{propgenquasi}, we directly prove the spectral decomposition of $\cL_k$ through studying Lasota-Yorke inequality on $\cL_k$ in Lemma~\ref{lem:LY} and applying Ionescu Tulcea Marinescu Theorem \cite{ITM1950} on it. We start with two technical lemma inspired from Lemma 2.7 and 2.9 from \cite{BKL22} but adapted to our settings.

\begin{lemma}\label{lemlomega}
Let $\lambda_\omega:=\min \{\lambda_{\omega,i},i\leq M(\omega)\}$ and $\Lambda_\omega:=\max\{\lambda_{\omega,i}, i\leq M(\omega)\}$.
For all $k,r\in \bN_0$, \(\omega\in \Omega\),
\begin{align}\label{eq:derivl}
(L_{\omega,k}h)^{(r)}=L_{\omega,k+r}(h^{(r)}).
\end{align}
 Furthermore, for any $k, r\in \bN_0$, there are constants $\Gamma_{\omega,0}:=\Lambda_{\omega}>0$ and $\Gamma_{\omega,k}=\lambda_\omega^{-(k-1)}$  such that the Lasota-Yorke inequality holds, that is
    \begin{align}\label{eqlylomega}
        \norm{L_{\omega,k}h}_r&\leq \Gamma_{\omega,k} \norm{h}_r\\
        \norm{L_{\omega,k}h}_r&\leq \lambda_\omega^{-(k+r-1)}\norm{h}_{r-1}+\Gamma_{\omega,0} \norm{h}_r
    \end{align}
     The first inequality also holds in the case \(r= 0\).
\end{lemma}

\begin{proof}
According to Lemma \ref{lemextpartition}, for $h\in W_r(\tilde \cI)\subset W_r(\cI_\omega)$ and Lemma 2.7 from \cite{BKL22} applies proving \eqref{eq:derivl}. 
As for inequality \eqref{eqlylomega}, we deduce from Lemma \ref{lemextpartition} and Lemma 2.9 in \cite{BKL22},
\begin{align*}
 \norm{L_{\omega,k}h}_r&\leq \norm{L_{\omega,k}h}_{W_r(\cI_\omega)}\\
 &\leq \Gamma_{\omega,k} \norm{h}_{W_r(\cI_\omega)}\\
  &\leq \Gamma_{\omega,k} \norm{h}_r,
\end{align*}
and the same holds for the Lasota-Yorke inequality :
\begin{align*}
    \norm{L_{\omega,k}h}_r&=\norm{L_{\omega,k}h}_{W_r(\cI_\omega)}\\
    &\leq \lambda_\omega^{-(k+r-1)}\norm{h}_{W_{r-1}(\cI_\omega)}+\Gamma_{\omega,0} \norm{h}_{W_r(\cI_\omega)}\\
    &\leq \lambda_\omega^{-(k+r-1)}\norm{h}_{r-1}+\Gamma_{\omega,0} \norm{h}_r.
\end{align*}

\end{proof}

We can now lift the previous Lemma to the annealed operator.

\begin{lemma}\label{lem:to_deriv}
    For all \(k,r\in\bN_0\), \(h\in W_r\) and \(l\in\{0,\ldots,r\}\),
    \[
    (\cL_k h)^{(l)} = \cL_{k+l} h^{(l)}.
    \]
\end{lemma}
\begin{proof}
    Fix \(k,r\in\bN_0\), without loss of generality, we assume $h\in W_r(\tilde \cI)$ with $\norm{h}_{W_r}=1$. The claimed equality holds trivially in the case \(l = 0\).
 Thanks to Lemma~\ref{lemlomega}, for any $\omega \in \Omega$ and any $l\leq r$, 
    \begin{equation}\label{eq:L_deriv}
        (L_{\omega,k}h)^{(l)}=L_{\omega,k+l}(h^{(l)}).
    \end{equation}
    
    Consider the continuous derivative operator $D : W_r(\tilde \cI) \to W_{r-1}(\tilde \cI)$, then we can make use of relation \eqref{eq:intosum} from the proof of Proposition \ref{propgenquasi} and its associated notations:

\begin{align*}
    \norm{D \left(\int L_{\omega,k} h d\bP\right)- \int D(L_{\omega,k} h) d\bP}_{k-1}&=\norm{D \left(\int L_{\omega,k} h d\bP\right) -\sum_i D L_i(h) \mu(F_i))}_{k-1}\\
    &+\norm{ \int DL_{\omega,k} h d\bP -\sum_i DL_i(h) \mu(F_i))}_{k-1}\\
    &\leq \norm{D}_{W_k\to W_{k-1}}\norm{ \int L_{\omega,k} h d\bP -\sum_i L_ih \mu(F_i))}_{k}+2\ve\\
    &\leq 4\ve.
\end{align*}
Where the inequality $\norm{ \int DL_{\omega,k} h d\bP -\sum_i DL_i(h) \mu(F_i))}_{k-1}\leq 2\ve \norm{Dh}_{k-1}\leq 2\ve$ is a consequence of relation \eqref{eq:intosum} and the fact that 
$\int_{X^{-1}(K_\ve^c)} \|L_{\omega,k}\|_rd\bP(\omega)\leq \ve$.
Thus,
\begin{equation}
    D(\cL_k h)=\cL_k(Dh).
\end{equation}\label{eq:comderiv}
    Consequently, using equation \eqref{eq:L_deriv}, we get for any \(l\in\bN_0\), 
    \[
    (\cL_k h)^{(l)} = \cL_{k+l}(h^{(l)}).
    \]
\end{proof}

We can now adapt the Lasota Yorke inequality from Lemma 2.9 from \cite{BKL22} to the operator $\cL_k$.
\begin{lemma}\label{lem:LY}
There are constant $\Gamma_0:=\int_\Omega \Lambda_\omega d\bP(\omega)>0$ and $\Gamma_k=\lambda^{-(k-1)}$ (Recall that $\lambda:=\inf_{\omega\in \Omega}\lambda_\omega >0$ and $\Lambda_\omega:=\norm{T_\omega'}_\infty$) such that the Lasota-Yorke inequality holds,
    \begin{align}
        \norm{\cL_kh}_r&\leq \Gamma_k \norm{h}_r \label{eqlem:LY1}\\
        \norm{\cL_kh}_r&\leq \lambda^{-(k+r-1)}\norm{h}_{r-1}+\Gamma_0 \norm{h}_r
    \end{align}
     The first inequality also holds in the case \(r= 0\).
\end{lemma}

\begin{proof}
We first consider the case when \(k\in\bN\). Let \(h\in W_r\), by definition of \(\norm{\cdot}_r\), Lemma~\ref{lem:to_deriv}, and the fact that $\lambda\leq |T_\omega'(x)|$ for all $\omega\in \Omega$ and $x\in I$,
\begin{align}
    \norm{\cL_k h}_r &= \sum_{l=0}^r \int_{\tilde \cI} |(\cL_k h)^{(l)}(x)|\ dx =                           \sum_{l=0}^r \int_{\tilde \cI} |\cL_{k+l} h^{(l)}(x)|\ dx\nonumber\\
                     &= \sum_{l=0}^r \int_{\tilde \cI} \left|\int_\Omega L_{\omega,k+l} h^{(l)}(x)\ d\bP(\omega)\right|\ dx\nonumber\\
                     &\leq \sum_{l=0}^r \int_{\tilde \cI} \left(\int_\Omega |L_{\omega,k+l} h^{(l)}(x)|\ d\bP(\omega)\right)\ dx\nonumber\\
                     &= \sum_{l=0}^r \int_{\Omega} \int_{\tilde \cI} |L_{\omega,k+l} h^{(l)}(x)|\ dx\ d\bP(\omega)\nonumber\\
                      &\leq \sum_{l=0}^r \int_{\Omega} \int_{\tilde \cI} |L_{\omega,1} \frac{h^{(l)}(x)}{|T_\omega'(x)|^{k+l-1}}|\ dx\ d\bP(\omega)\nonumber\\
                     &\leq \sum_{l=0}^r \lambda^{-(k+l-1)} \int_{\Omega} \int_{\tilde \cI} L_{\omega,1} |h^{(l)}(x)|\ dx\ d\bP(\omega)\label{eq:lypk}
\end{align}
    By change of variables, \(\int_{\tilde \cI} L_{\omega, 1}|h^{(l)}(x)|\ dx = \int_{\tilde \cI} |h^{(l)}(x)|\ dx\), thus for all \(r\in\bN_0\),
\begin{equation*}
    \norm{L_{\omega,k} h}_r \leq  \lambda^{-(k-1)} \sum_{l=0}^r\int_{\tilde \cI} |h^{(l)}(x)|\ dx.
\end{equation*}
    For \(\Gamma_k = \lambda^{-(k-1)}\), we have the first inequality.
    When \(r\geq 1\), equation \eqref{eq:lypk} implies that 
\begin{align*}
    \norm{L_{\omega,k} h}_r & \leq \lambda^{-(k+r-1)} \int_{\tilde \cI} |h^{(r)}(x)|\ dx +          \lambda^{-(k-1)} \sum_{l=0}^{r-1}\int_{\tilde \cI} |h^{(l)}(x)|\ dx\\
                     & \leq \lambda^{-(k+r-1)}\norm{h}_r + \lambda^{-(k-1)} \norm{h}_{r-1},
\end{align*}    
    and hence the second estimate.                                                                                                   

    Now we consider the case when \(k=0\). Note that for any \(h\in W_r\),
   \begin{align*}
       \int_{\tilde \cI} |\cL_0  h|(x)\ dx &= \int_{\tilde \cI} |\int_\Omega L_{\omega,0} h (x) d\bP(\omega)| dx\\
                                  &= \int_{\tilde \cI} |\int_\Omega L_{\omega,1}T_\omega' h (x) d\bP(\omega)| dx\\
                                  & \leq \int_{\tilde \cI} \int_\Omega L_{\omega,1}|T_\omega'|\ | h| (x) d\bP(\omega) dx\\
                                  &\leq   \int_\Omega  \Lambda_{\omega} \int_{\tilde \cI} L_{\omega,1} | h| (x)  dx d\bP(\omega)\\
                                   &\leq \int_\Omega \Lambda_{\omega}\int_{\tilde \cI}  | h| (x) dx d\bP(\omega)\\
                                  &\leq \int_\Omega \Lambda_{\omega}d\bP(\omega)\int_{\tilde \cI} |  h|(x)\ dx\\
                                   &\leq \Gamma_0\int_{\tilde \cI} |  h|(x)\ dx.
   \end{align*}
    Now using Lemma~\ref{lem:to_deriv}, we get
    \begin{align*}
        \norm{\cL_0h}_r = \sum_{l=0}^r\int_{\tilde \cI}|(\cL_0h)^{(l)}(x)|\ dx 
                        = \sum_{l=0}^r\int_{\tilde \cI}|\cL_l h^{(l)}(x)|\ dx. 
    \end{align*}
    According to \eqref{eq:lypk}, for all \(r\in\bN_0\),
    \[
    \norm{\cL_0h}_r\leq \sum_{l=0}^r \lambda^{-l}\int_{\tilde \cI}|\cL_0h^{(l)}(x)|\ dx\leq  \Gamma_0 \sum_{l=0}^r\int_\cI|h^{(l)}(x)|\ dx
    \]
    which proves the first inequality. Now, for \(r\geq 1\),

    \begin{align*}
        \norm{\cL_0h}_r & = \int_{\tilde \cI}|(\cL_0h)^{(r)}(x)|\ dx + \sum_{l=0}^{r-1} \int_\cI|(\cL_0h)^{(l)}(x)|\ dx\\
                        & \leq \int_{\tilde \cI}|\cL_r h^{(r)}(x)|\ dx + \sum_{l=0}^{r-1} \int_\cI|\cL_l h^{(l)}(x)|\ dx\\
                        & \leq \lambda^{-(r-1)} \int_{\tilde \cI}|\cL_1 h^{(r)}(x)|\ dx + \norm{\cL_0h}_{r-1}\\
                        &\leq \lambda^{-(r-1)} \int_{\tilde \cI}| h^{(r)}(x)|\ dx + \Gamma_0 \norm{h}_{r-1}
    \end{align*}

\end{proof}

We are now able to complete the proof of Theorem~\ref{thm:proj} using item~(\ref{item2:propgenquasi}) from Proposition~\ref{propgenquasi}.
    The first inequality of Lemma~\ref{lem:LY} helps deducing that the spectral radius of the operator \(\cL_k\) is bounded by \(\Gamma_k\). Additionally, using the compact injection of Sobolev spaces \(W_r\) inside \(W_{r-1}\) for all \(r\in\bN\) and the Hennion-Nussbaum theory \cite{DKL2021}, we get that the Lasota-Yorke inequalities given by Lemma~\ref{lem:LY} implies that the essential spectral radius of \(\cL_k\) is bounded by \(\lambda^{-(k+r-1)}\).
Thus we obtain a spectral decomposition of $\cL_k : W_r\to W_r$ as 
\begin{align*}
    \cL_k:=\Pi_k+N_k,
\end{align*}
where $\Pi_k$ a finite rank operator and $N_k$ a linear operator such that there is $C_k>0$, for all $n\in \bN$, $\|N_k^n\|\leq C_k\lambda^{-(k+r-1)n}$. Thus the remaining of the proof consists of proving that the required item~(\ref{item2:propgenquasi}) in Proposition~\ref{propgenquasi} holds in our case. 

\begin{proof}[Proof of Theorem~\ref{thm:proj}]
Let $k,r\in \bN$, by Lemma \ref{lem:LY}, we proved that $\cL_k$ has essential spectral radius no greater than $\lambda^{-(k+r-1)}$. Now, for \(\delta>0\) arbitrarily small, we can focus on describing the point spectrum that lies in 
$$
\cH_{k,r,\delta}:=\{\nu\in \sigma_{W_r}(\cL_k),|\nu|>\lambda^{-(k+r-1)}+\delta\}
$$ 
and its associated eigenspaces.
Let $\nu \in \cH_{k,r,\delta}$,
    recall from Lemma \ref{lem:to_deriv} that the derivative $D: h\in W_{r}\mapsto h'\in W_{r-1}$ follows the following semi conjugacy,
    $$
    D^l\circ (\cL_k -\nu)^m h= (\cL_{k+l} -\nu)^m\circ D^l h,
    $$
for $\nu\in \bC$, the associated eigenspace $E_k(\nu)$ for the eigenvalue $\nu$ of $\cL_k$ is linked to $E_{k+l}(\nu)$ as follows,
\begin{equation}\label{eq:deriv}
    D^lE_k(\nu)\subset E_{k+l}(\nu).
\end{equation}
  
Thus for any $l\in \bN$, $E_k(\nu) \not\subset  \fB_l(\tilde \cI)$ iff\footnote{recall that \(\fB_r(\tilde \cI)\) is  the space of piecewise polynomial functions of degree \(r\) on each interval $I_i\in \tilde \cI$ which is invariant by $\cL_k$ as a direct application of its expression \eqref{eq:transfer_operator}.} $E_{k+l}(\nu)\neq \{0\}$. Now notice from  equation \eqref{eqlem:LY1} of Lemma \ref{lem:LY} that $E_{k+l}(\nu)\neq \{0\}$ only if $\nu\leq \lambda^{-(k+l)}$. Thus for any $\nu \in \cH_{k,r,\delta}$, $E_{k+r}(\nu)=\{0\}$ and thus $E_{k}(\nu)\subset \cB_{r}(\tilde \cI)$.
Then if we fix $\delta:= \lambda^{-(k+r)}$, we obtain the same decomposition as \eqref{eqthmegenquasi} from Proposition \ref{propgenquasi},
\begin{align}\label{eq:Lquasicon}
    \cL_k=\sum_{\nu \in  \cH_{k,r,\delta}}\cL_k\circ\Pi_\nu + N_k,
\end{align}
with $N_{k}:=\cL_k\circ (Id-\Pi_{k})$ satisfying, for some $c_{k,r}>0$ and all $n\in \bN$, $\norm{N_{k}^n}_{r}\leq 2c_{k,r}\lambda^{-n(k+r-1)}$. Furthermore we obtained for any $r\in \bN$ that $E_k(\nu)\subset C^\infty(\tilde \cI,\bR)\subset \bigcap_{r}W_r(\tilde \cI)$. Thus the required assumptions of the second point of Theorem \ref{eqthmegenquasi} are fulfilled and according to it, we obtain the following decorrelation property : for any $\ve>0$, there exists \(r\in\bN\) large enough such that for any $k,i,j\in \bN$ there is some bilinear map $C_{i,j,k} : W_r(\tilde \cI)\times W_r(\tilde \cI)\mapsto \mathbb{C}$ such that for any $\phi,\psi\in C^\infty(\tilde \cI,\bR)$,
\begin{align}
\int \cL_k^n \phi \cdot \psi d\mu = \sum\limits_{|\lambda_i|\geq\ve}\sum\limits_{j\leq N_i}\lambda_i^n n^j C_{i,j,k}(\phi,\psi) + o(\ve^n),       
\end{align}
concluding the result.
\end{proof}

As a result from the proof of Theorem \ref{thm:proj} recall that for any $\nu \in \cH_{k,r,\delta}$, $E_k(\nu)\subset \fB_r(\tilde \cI)$. And since $\fB_r(\tilde \cI)$ is invariant by $\cL_k$,
one can construct the Jordan form of the annealed operator $\cL_k$ acting on the finite dimensional space $\fB_r(\tilde \cI)$ and thus read the main eigenvalues of the discrete spectrum above the radius $\lambda^{-(k+r-1)}$ on it.  
We will now present rigorously how this Jordan form is expressed according to the Jordan form of the matrices $B_k$ introduced in equation \eqref{eq:bkformula}.

We construct a basis of monomials of $\fB_r(\tilde \cI)$ and identify  $\fB_r(\tilde \cI)$ with  \((\bR^{N})^{(r+1)}\) where $N:=\# \tilde \cI$ as follows.
For $(a^0,a^1,\ldots , a^r) \in (\bR^{N})^{(r+1)}$ when denoting $a^{l} = (a_{1}^{l},a_{2}^{l},\ldots,a_{N}^{l})$ for each $l$,
let
$\cJ : \bR^{N(r+1)} \to \fB_r(\tilde \cI)$,
\begin{align}
    \cJ(a^0,\ldots , a^r): x
    \mapsto
    \sum_{l=0}^{r}  x^{l} \sum_{j=1}^{N}  a^l_j 1_{I_j}(x).
\end{align}
Such a map $\cJ$ is an isomorphism between  $\fB_r(\tilde \cI)$ and \((\bR^{N})^{(r+1)}\). We also recall from \eqref{eq:transfer_operator} that for $h\in W(\tilde \cI)$, $ L_{\omega,k} (h)=\sum_{1\leq i \leq n} 1_{I_i}\sum_{1\leq j\leq M(\omega)}B_{\omega,k}[i,j] h\circ g_{j,\omega}$ where $g_{j,\omega}:=(T_{\omega}|_{I_J})^{-1}$ is an affine map over $T_\omega(I_j)$ and thus can be written as $g_{j,\omega}(x)=\lambda_{\omega,j}^{-1}(x-p_{\omega,j})$, thus for $\mathbf{a}:=(0,\dots,0,a^l,0,\dots,0)$ and $x\in I$

\begin{align}
        \begin{aligned}
            L_{\omega,k} (\cJ(\mathbf{a}))(x)
             & =\sum_{1\leq i\leq N}\sum_{1\leq j\leq M(\omega)} 1_{I_i}(x) B_{\omega,k} [i,j]a_j^l \left(\lambda_{\omega,j}^{-1}(x-p_{\omega,j})\right)^{l}
            \\
             & =\sum_{1\leq i\leq N}\sum_{1\leq j\leq M(\omega)} 1_{I_i}(x) B_{\omega,k+l} [i,j]a_j^l x^{l} + \rho_{\omega}(x),\label{eq:Lkxa}
        \end{aligned}
\end{align}
with $\rho_{\omega}\in \fB_{l-1}(I)$ ($\rho_\omega$ is the rest in the expansion in $\lambda_{\omega,j}^{-1}(x-p_{\omega,j})$).

Now to express $\cL_k$ on $\fB_r(\tilde \cI)$, we need to aggregate the different matrices $B_{\omega,k}$ involved. These matrices having different number of columns, one can aggregate them as follows using  the fact that for each $\omega\in \Omega$, $\cI\preceq \cI_\omega$ :
\begin{align}
        \label{eq:Lkxa}
        \begin{aligned}
            L_{\omega,k} (\cJ(\mathbf{a}))(x)
             & =\sum_{1\leq i\leq N}\sum_{1\leq j\leq M(\omega)} 1_{I_i}(x) B_{\omega,k+l} [i,j]a_j^l x^{l} + \rho_{\omega}(x)
            \\
             & =\sum_{1\leq i\leq N}\sum_{1\leq j\leq N}\sum_{\{m,I_{\omega,m}\subset I_j\}} 1_{I_i}(x) B_{\omega,k+l} [i,m]a_j^l x^{l} + \rho_{\omega}(x).
        \end{aligned}
\end{align}

Thus taking the integral, we obtain for $\cL_k$,
\begin{align}
    \cL_k (\cJ(\mathbf{a}))(x)&=\int_\Omega L_{\omega,k}(\cJ(\mathbf{a}))(x) d\bP(\omega)\\
     & =\sum_{i,j} \int_\Omega 1_{I_i}(x) \sum_{\{m,I_{\omega,m}\subset I_j\}}B_{\omega,k+l} [i,m]a_j^l x^{l} d\bP(\omega) + \int_\Omega \rho_{\omega}(x)d\bP(\omega)\notag\\
     &=\sum_{i,j} 1_{I_i}(x) a_j^l x^{l} \int_\Omega \sum_{\{m,I_{\omega,m}\subset I_j\}} B_{\omega,k+l} [i,m] d\bP(\omega) + \int_\Omega \rho_{\omega}(x)d\bP(\omega) \label{eq:ann:Lkxa}.
\end{align}
  In other words $ \cL_k (\cJ(\mathbf{a}))(x)=\cJ\left(\mathbf{b} \right)+\int_\Omega \rho_{\omega}(x)d\bP(\omega)$ where $\mathbf{b}:=(0,\cdots,\left(B_{k+l}a^l \right),\cdots,0)$ with $B_{k+l}$ being the $N\times N$ matrix given by 
  \begin{align}\label{eqtransmat}
  B_{k+l}[i,j]:=\int \sum_{\{m,I_{\omega,m}\subset I_j\}}B_{\omega,k+l}[i,m] d\bP(\omega)    
  \end{align}
   Thus the expression of $\cL_k$ in the basis of $\fB_r(\tilde \cI)$ we introduced, is a matrix $\cT_k:= \cJ^{-1}\circ\cL_k\circ \cJ$ which is lower triangular with blocks of the form $B_{k+l}$ on the diagonal.  
  the corresponding matrix on that basis is lower triangular with diagonal blocks corresponding to the matrix $B_{k+l}$ as follows :
\begin{align}\label{eq:matrixrep}
        \cT_{k,r} = 
        \begin{pmatrix}
            B_{k}   &         &        & 0       \\
            F_{1,0} & B_{k+1} &        &         \\
            \vdots  & \vdots  & \ddots &         \\
            F_{r,0} & F_{r,1} & \ldots & B_{k+r}
        \end{pmatrix}.
\end{align}

Such matrix can thus be reduced into a Jordan form $\cF_{k,r}:=P_{k,r}^{-1}\cT_{k,r}P_{k,r}$ (with $P_{k,r}$ the transition matrix) whose eigenvalues lie in the set of eigenvalues of the diagonal blocks $B_{k+l}$. This decomposition thus generates the Ruelle-Pollicott resonances.\\

\begin{corollary}
    Under the condition of Theorem \ref{thm:proj}, the random family of dynamical systems $(T_\omega)_{\omega\in\Omega}$ is mixing and admits a stationary measure.
\end{corollary}

\begin{proof}
Fix \(r\ge 1\). For \(k=1\), the matrices \(B_{\omega,1}\) are non-negative and left stochastic. Averaging over \(\Omega\), we obtain the matrix
\[
B_1:=\int_\Omega B_{\omega,1},d\mathbb P(\omega),
\]
which is again a non-negative left stochastic matrix. Since the family is transitive, \(B_1\) is irreducible. Hence, by the Perron--Frobenius theorem (see, e.g., \cite[Section 9.3]{BoGo97}), there exists a positive diagonal matrix \(D\) such that \(DB_1D^{-1}\) is stochastic. Consequently, \(1\) is a simple eigenvalue of \(B_1\), and all other eigenvalues satisfy \(|z|<1\).

For \(2\le k\le r+1\), the matrices \(B_k\) are non-negative and satisfy
\[
\sum_{i=1}^{N} B_k[i,j]\le \lambda^{k-1}<1,
\]
for every column \(j\). Therefore, \(r(B_k)\le |B_k|_1 \le \lambda^{k-1}<1\), so the spectrum of \(B_k\) is contained in a disc of radius strictly smaller than \(1\).

By the matrix representation \eqref{eq:matrixrep}, the action of \(\mathcal L_0\) on \(\mathfrak B_r(\tilde{\mathcal I})\) is upper triangular, with diagonal blocks given by \(B_1,\dots,B_{r+1}\). Hence the only spectral value of modulus one arising from this finite-dimensional part is the simple eigenvalue \(1\) coming from \(B_1\). Combining this with the quasi-compactness of \(\mathcal L_0\) established in \eqref{eq:Lquasicon}, we conclude that \(1\) is a simple eigenvalue of \(\mathcal L_0\) and that no other eigenvalue lies on the unit circle. Therefore \(\mathcal L_0\) has a spectral gap. Standard spectral arguments then imply that the random system is mixing and admits a unique stationary probability measure \(\mu\ll m\), whose density is the positive eigenfunction corresponding to the eigenvalue \(1\).
\end{proof}

\subsection{Example}\label{sec:eg}
In this simple example of annealed random dynamical system, we highlight the fact that Ruelle-Pollicott resonances for the annealed operator $\cL_k$ in general are not simply the average of eigenvalues, that is $\int_\Omega \lambda_{\omega,i}d\bP(\omega)$ where the $\lambda_{\omega,1}\geq \lambda_{\omega,2}\geq \dots  $ stand for the eigenvalues of $L_{\omega,k}$ counted with multiplicity and ordered starting from the top eigenvalue. We will see however in the next section that random families of hyperbolic Blaschke product are families whose Ruelle-Pollicott resonances have this property. \\

To set up our example we choose two maps $T_1 : x\in I\mapsto \beta x \,mod 1$ and $T_2 : I\mapsto I  $, for some\footnote{For example, take $x_0=\frac 1 4$ and $\beta=4$ (actually any integer $\beta$ would work).} $x_0:=\frac{m}{\beta}$ with $m\in \bN$ and $\frac{m}{\beta}<\frac 1 2$, defined as

\begin{align*}
    T_2(x)=\left\{ \begin{array}{r l c}
        \frac{1-x_0}{x_0}x+x_0 & \textit{ if }x\in [0,x_0]  \\
         \frac{1}{1-x_0}(x-x_0) & \textit{ otherwise } \\
    \end{array}\right.
\end{align*}
with $\alpha_1^{-1}=\frac{1-x_0}{x_0}$ and $\alpha_2^{-1}=\frac{1}{1-x_0}$. As a common reference partition between the two maps, one can choose $\cI:=\{(0,x_0),(x_0,1)\}$. 
Taking the Markov partition $\cI_2$ for $T_2$ as $\cI$, this map has the following transition matrix $A_2=\left(\begin{array}{r l c}
    0 & 1 \\
     1& 1
\end{array}\right)$ and the corresponding block $B_{2,k}:=\left(\begin{array}{r l c}
    0 & \alpha_2^k \\
     \alpha_1^k& \alpha_2^k
\end{array}\right)$.
As for the $\beta$-map $T_1$, the corresponding matrix is obtained from \eqref{eqtransmat} as 
$B_{1,k}:=\left(\begin{array}{r l c}
    x_0\beta^{-(k-1)} & (1-x_0)\beta^{-(k-1)} \\
     x_0\beta^{-(k-1)}& (1-x_0)\beta^{-(k-1)}
\end{array}\right)$.

On $\Omega=\{1,2\}$ we fix the probability $\bP$ as follows, $\bP(\{1\})=v$ and $\bP(\{2\})=1-v$, for some \(0<v<1\). 

According to the equation \eqref{eq:matrixrep}, the annealed operator $\cL_k: \fB_r(\cI)\mapsto \fB_r(\cI)$ has diagonal blocks of the following form,
\begin{align*}
    \tilde B_k&=(vB_{1,k}+(1-v)B_{2,k})\\
    &=\left(\begin{array}{r l c}
    vx_0\beta^{-(k-1)} & (1-v)\alpha_2^k+v(1-x_0)\beta^{-(k-1)} \\
     (1-v)\alpha_1^k+vx_0\beta^{-(k-1)}& (1-v)\alpha_2^k+v(1-x_0)\beta^{-(k-1)}
\end{array}\right)
\end{align*}

Notice that generically $B_{1,k}$ and $B_{2,k}$ can be diagonalized, however if we denote $\lambda_1$ the leading eigenvalue of $B_{1,k}$ and $\lambda_2$ for $B_{2,k}$, $\frac{\lambda_1+\lambda_2} 2$ would be an eigenvalue for $\tilde B_k$ only if both matrices have the same eigenvector for their respective eigenvalues. Reasoning the same way for the least eigenvalue of respectively $B_{1,k}$ and $B_{2,k}$, we deduce that the spectrum of $\cL_k$ is given by $\int_\Omega \lambda_{\omega,i}d\bP(\omega)$ only if both matrices $B_{1,k}$ and $B_{2,k}$ commute, which they don't in general.
To check this, one can for example choose the following parameters $v=\frac 1 2$, $\beta=4$ and $x_0=\frac 1 4$ : then $\alpha_1^{-1}=3$ and $\alpha_2^{-1}=\frac 4 3$, thus 

\begin{align*}
    \tilde B_k&=\frac 1 2 (B_{1,k}+B_{2,k})\\
    &=\frac 1 2 \left(\begin{array}{r l c}
    4^{-k} & (\frac 3 4)^k+(\frac 3 4)4^{-(k-1)} \\
     (\frac 1 3)^k+4^{-k}& (\frac 3 4)^k+(\frac 3 4)4^{-(k-1)}
\end{array}\right)\\
&=\frac 1 2 \left(\begin{array}{r l c}
    4^{-k} & \frac{3^{k}+3}{4^k} \\
     3^{-k}+4^{-k}& \frac{3^{k}+3}{4^k}
\end{array}\right).
\end{align*}

A direct computation of the respective eigenvalues of the blocks $\tilde B_k$ gives
\begin{align*}
    \lambda_{i,k}:=\frac{\dfrac{3^k+4}{4^k}+(-1)^i\left(\dfrac{3^k+2}{4^k}+3^{-k}\right)}{2}.
\end{align*}
These eigenvalues (which are the Ruelle-Pollicott resonances for the annealed operator) are all distincts, they thus have multiplicity $1$.
One can compare these to the averaged eigenvalues $\frac{ \lambda_{i,k}^{(2)}+ \lambda_{i,k}^{(1)}}{2}$ of the two respective transfer operators corresponding to the transformation $T_1$ and $T_2$. Contrary to the Blaschke product case (see Theorem \ref{thm:Bl_spectrum}) the two families do not genrally coincide. To see that in our example, notice that the eigenvalues $\lambda_{i,k}^{(1)}$ for $B_{1,k}$ are respectively given by
\begin{align*}
    \lambda_{1,k}^{(1)}=4^{-(k-1)}, \text{ and } \lambda_{0,k}^{(1)}=0. 
\end{align*}
As for $B_{2,k}$, the eigenvalues $\lambda_{i,k}^{(2)}$ are given by
\begin{align*}
    \lambda_{i,k}^{(2)}=\frac{\left(\frac34\right)^k+(-1)^i\frac{\sqrt{3^{2k}+4^{k+1}}}{4^k}}{2}
=\frac{3^k+(-1)^{i}\sqrt{3^{2k}+4^{k+1}}}{2\cdot 4^k}.
\end{align*}
Thus the eigenvalues (Ruelle-Pollicott resonances) of $\tilde B_k$ do not coincide with the average $\frac{ \lambda_{i,k}^{(2)}+ \lambda_{i,k}^{(1)}}{2}$ of eigenvalues from $B_{1,k}$ or $B_{2,k}$.

\section{Some compact annealed operators and their resonances}
\subsection{Inverse resonance problem featuring Blaschke products}\label{secblaschke}

\subsubsection{Settings}
Let $\bT^2:=\{z=(z_1,z_2)\in \bC^2, |z_1|=|z_2|=1\}$ be the $2$-torus we define a class of analytic hyperbolic diffeomorphisms: a Blaschke product $T_\omega :\bT^2\mapsto \bT^2$ defined for $\omega \in \bC$ with $|\omega|<1$ and $z=(z_1,z_2)\in \bT^2$ by
\begin{align}\label{blaschke_stand}
    T_{\omega}(z_1,z_2)=\left(z_1\frac{z_1-\omega}{1-\bar{\omega}z_1}z_2,\frac{z_1-\omega}{1-\bar{\omega}z_1}z_2\right).
\end{align}

\begin{remark}
    In the whole Section \ref{secblaschke} we focus on the family of Koopman operators $U_\omega :f\mapsto f\circ T_\omega$ and the associated annealed operator $\cU$ (defined in \ref{def:cU}). However notice that for any $\omega \in \bC$, $|\omega|<1$, $T_\omega^{-1}$ is also a Blaschke product and the family of transfer operators as defined in formula \ref{eqoptrans} is of the form $L_\omega : f\mapsto f\circ T_\omega^{-1}$ thus all the following results also hold for the associated Transfer operator $\cL: f\mapsto \int f\circ T_\omega^{-1}d\bP(\omega) $.
\end{remark}  

\subsubsection{Hilbert Space and Transfer operator}
We consider, for $a>0$, the space $(H_a,\norm{\cdot}_{a})$,  formalized in the work of \cite{SBJ17}, as a completion by some norm $\norm{\cdot}_a$ of the space $\cB$ of Laurent polynomials,
$$
\cB:=\{f:\bT^2\mapsto \bC, f(z)=\sum_{|n|\leq N}f_nz^n, \textit{ with }f_n\in \bC, N\in \bN\}.
$$
For $n:=(n_1,n_2)\in \bZ^2$, we define $n_u = n_u(n_1,n_2)$ and $n_s= n_s(n_1,n_2)$ respectively as
\begin{align*}
    n_{u}&:=(\omega_u-1)n_1+n_2\\
    n_s&:=(\omega_s-1)n_1+n_2,
\end{align*}
where $\omega_u:=\frac{1+\sqrt{5}}{2}$ and $\omega_s:=\frac{1-\sqrt{5}}{2}$.
We can now define $\norm{\cdot}_a$ which derives from the scalar product $<\cdot,\cdot>_a$ defined for $f,g\in \cB$ as
\begin{align*}
    <f,g>_a=\sum_{|n|\leq N}f_n\bar{g_n}\exp(-2a|n_u|+2a|n_s|).
\end{align*}

The space $(H_a,\norm{\cdot}_{a})$ is then the completion of $(\cB,\norm{\cdot}_{a})$ which happens to be a Hilbert space.

\begin{defin}\label{def:cU}
Let $U_\omega : H_a \mapsto H_a$ be the Koopman operator associated to $T_\omega$, $f\in H_a \mapsto f\circ T_\omega$,  we define the annealed Koopman operator $\cU : H_a \to H_a$ with the base space \(\Omega=\{\omega\in\bC: |\omega|<1\}\), as
\[
\cU:=\int_\Omega U_\omega d\bP(\omega),
\]
provided the probability measure on \(\Omega\) is such that \(\int_\Omega \norm{U_\omega}_a\ d\bP(\omega)<\infty\). 
\end{defin}

  \begin{remark}
 In general, one can choose to consider any other base space \(\Omega\) instead of $\{\omega\in \bC,|\omega|<1\}$ and rather define the Blaschke product as 
\[
    T_{\omega}(z_1,z_2)=\left(z_1\frac{z_1-b_\omega}{1-\bar{b_\omega}z_1}z_2,\frac{z_1-b_\omega}{1-\bar{b_\omega}z_1}z_2\right)
\]
where the map \(\omega\in\Omega\mapsto b_\omega\in \{z\in\bC: |z|<1\}\) is measurable. This model can always be reduced to the one stated at the beginning of this section ( see formula \eqref{blaschke_stand}) by a change of variable taking $\bP_b:=(\omega\mapsto b_\omega)*\bP$ as our probability measure on the space $\{z\in \bC,|z|<1\}$ :
\begin{align*}
    \int_{\omega\in \Omega}U_{b_\omega} d\bP(\omega)=\int_{|z|<1 }U_zd\bP_b(z).
\end{align*}
The integrability condition on $(U_{b_\omega})_{\omega \in \Omega}$ reduces to \(\int_\Omega \norm{U_{b_\omega}}_a\ d\bP(\omega)<\infty\).
    
\end{remark}

\noindent

\begin{theorem}\label{thm:Bl_spectrum}
Let \(\mu\) be the volume measure on \(\bT^2\).
    The annealed Koopman operator $\cU:H_a \to H_a$ is a well-defined compact operator and for any \(a>0\), $\cU$ satisfies the following decomposition
    \[
\int f \cdot \cU^n g\ d\mu = \sum\limits_{|\lambda_i|\geq\ve}\sum\limits_{j\leq N_i}\lambda_i^n n^j c_{i,j}(f,g) + o(\ve^n),
\]
Furthermore, the spectrum is given by 
    \begin{equation}\label{eq:Bl_spectrum}
         \sigma(\cU) = \{\int_\omega (-\omega)^nd\bP(\omega) : n\in\bN\} \cup \{\int_\omega (-\bar\omega)^nd\bP(\omega) : n\in\bN\}\cup \{0,1\}.
    \end{equation}
    Each non-zero element $\lambda\in \sigma(\cU)$ of the spectrum is an eigenvalue, the algebraic and geometric multiplicity $N$ of which coincides with the number $N:=|\{n\in\bN,\int_\omega (-\omega)^nd\bP(\omega)=\lambda\}|+|\{n\in\bN,\int_\omega (-\bar\omega)^nd\bP(\omega)=\lambda\}|$.
\end{theorem}

\begin{remark}\label{rem:ann_compact}\begin{enumerate}
    \item The volume measure \(\mu\) on \(\bT^2\) in the above theorem is invariant under \(T_\omega\) for every \(\omega\in\Omega\) (see \cite[Section 4]{SBJ17}).
    \item By \cite[Proposition~2.5]{SBJ17}, we get that, for any \(\omega\in \bC\) with \(|\omega|<1\) and any \(a>0\), the operator \(U_\omega : H_a\to H_a\) is compact. The compactness of \(U_\omega\) for any \(\omega \in \bC\) with \(|\omega|<1\) implies by Proposition~\ref{prop:gencomp} that the operator \(\cU : H_a\to H_a\) is compact. Thus $\cU$ has \textbf{full Ruelle-Pollicott resonances} on $H_a$ and we only have to prove the localization of spectral data in Theorem \ref{thm:Bl_spectrum}.
\end{enumerate}
\end{remark}

\subsubsection{Spectral Data}

    As shown in \cite{SBJ17}, the vectors $e_n(z):=z^n \exp({a|n_u|-a|n_s|})$ generate an orthonormal Hilbert basis of $H_a$.
For \(\omega\in\bC\) with \(|\omega|<1\), the action of the operator \(U_\omega\) on the orthonormal basis is given by 
\begin{align}\label{eq:otgrep}
 (U_\omega e_n)(z) &= \exp({a|n_u|-a|n_s|}) z_1^{n_1}z_2^{n_1+n_2}  \left(\frac{z_1-\omega}{1-\bar\omega z_1}\right)^{n_1+n_2}\\
 &= \exp({a|n_u|-a|n_s|})\sum\limits_{m\in\bZ^2}b_{m,n}^\omega z^m,
\end{align}
where the expansion coefficients of the Laurent series in a neighborhood of the unit torus are given by
\begin{equation}\label{eq:bmn}
    b_{m,n}^\omega = \delta_{m_2,n_1+n_2}M_{-m_1+2n_1+n_2}(\omega,n_1+n_2),
\end{equation}
with $\delta_{p,q}$ standing for  Kronecker's symbol (i.e \(\delta_{p,q}=1\) for \(p=q\) and 0 otherwise), and where we have introduced the shorthand
\begin{equation}\label{eq:M_def}
    M_l(\omega,k) = \int_{|\zeta|=1}\zeta^l\left(\frac{1-\omega/\zeta}{1-\bar\omega \zeta}\right)^k \frac{d\zeta}{2\pi \iota\zeta}
\end{equation}

for the expansion coefficient of a single Blaschke factor. Hence, using the definition of the norm, we obtain
\begin{align}\label{ko_basis}
    \norm{U_\omega e_n}^2_a &= \sum\limits_{m\in\bZ^2}\exp(-2a|n_u|-2a|n_s|)\delta_{m_2,n_1+n_2}|M_{-m_1+2n_1+n_2}(\omega,n_1+n_2)|^2\\
    &\times \exp(-2a|m_u|+2a|m_s|).
\end{align}

Using the above calculation, we obtain 
\begin{equation}\label{eq:cUen}
    (\cU e_n)(z) = \exp({a|n_u|-a|n_s|})\sum\limits_{m\in\bZ^2}b_{m,n} z^m.
\end{equation}
where 
\begin{equation}\label{eq:bmn}
    b_{m,n} = \int_\Omega b_{m,n}^\omega\ d\bP(\omega).
\end{equation}

To compute the spectrum of the operator \(\cU\), we will consider suitable matrix representations of projections of this compact operator to finite dimensional subspaces. Observe that, using formulae \eqref{eq:cUen} and \eqref{eq:bmn},  the matrix representations \(\Gamma\) of \(\cU\) with respect to the orthonormal basis \((e_n)_{n\in \bZ^2}\) is of the form
\begin{equation}\label{eq:Gamma}
    \Gamma_{m,n} = \langle \cU e_n, e_m \rangle_a = b_{m,n} \exp(a|n_u|-a|n_s|-a|m_u|+a|m_s|).
\end{equation}
    A simple calculation using relation \eqref{eq:M_def} and \cite[Lemma~2.3 (i), (iii), (iv)]{SBJ17} yields the following:
    \begin{align}
    \label{eq:bmn-neq}    & m_2\neq n_1 + n_2 :\quad && b_{m,n} = 0\\ 
     \label{eq:bmn-0}   & m_2 = n_1 + n_2 = 0 :  && b_{m,n} =  \delta_{m_1,n_1} \\ 
     \label{eq:bmn-geq}   & m_2 = n_1 + n_2 > 0 : && b_{m,n} = \begin{cases}
            0\ &\text{if}\ m_1<n_1\\
            \int_\Omega (-\omega)^{m_2}\ d\bP(\omega)\ &\text{if}\ m_1=n_1
        \end{cases}\\ 
     \label{eq:bmn-leq}   & m_2 = n_1 + n_2 < 0 : && b_{m,n} = \begin{cases}
            0\ &\text{if}\ m_1>n_1\\
            \int_\Omega (-\bar\omega)^{-m_2}\ d\bP(\omega)\ &\text{if}\ m_1=n_1.
        \end{cases} 
    \end{align}
    Following what is done in \cite[Section~3]{SBJ17}, we first consider the following order of the basis elements:
\[
     e_{0,0}, e_{1,0}, e_{0,1}, e_{-1,0}, e_{0,-1}, e_{2,0}, e_{1,1}, e_{0,2}, e_{-2,0}, e_{-1,-1}, e_{0,-2}, \ldots
\]
    which is arranged as a sequence in the order of increasing norm \(|n|\), with
groups of elements with the same norm traversed in anti-clockwise direction. We then consider the re-ordering as follows. Go along the sequence above from left to right, on encountering an element \(e_{n_1,n_2}\) with \(n_1n_2 < 0\), we move it to the left-most position of the current sequence, and then we continue this algorithm ton obtain the following order:

\begin{equation}\label{eq:basis}
\ldots, e_{1,-1}, e_{-1,1}, e_{0,0}, e_{1,0}, e_{0,1}, e_{-1,0}, e_{0,-1}, e_{2,0}, e_{1,1}, e_{0,2}, e_{-2,0}, e_{-1,-1},e_{0,-2},\ldots
\end{equation}

\begin{lemma}[Approximation by finite-rank truncations]\label{lem:finite-approx}
Let $(\Omega,\mathbb P)$ be a probability space and $H_a$ a Hilbert space with orthonormal basis $(e_n)_{n\in\mathbb Z}$. Let $\Pi_N$ denote the orthogonal projection onto \(\mathrm{span}\{e_n : |n|\le N\}\). 
Define the finite approximations
\[
\mathcal U_N := \int_\Omega \Pi_N \circ U_\omega \circ \Pi_N \, d\mathbb P(\omega).
\]

Then
\[
\norm{\mathcal U_N - \mathcal U}_a \to 0 \quad \text{as } N \to \infty.
\]
\end{lemma}

\begin{proof}
We write
\(
\mathcal U_N - \mathcal U
=
\int_\Omega (U_{N,\omega} - U_\omega)\, d\mathbb P(\omega),
\quad
U_{N,\omega} := \Pi_N \circ U_\omega \circ \Pi_N.
\)

Using \eqref{eq:conL} from Proposition \ref{prop:integomeg} in the appendix, we get
\[
\norm{\mathcal U_N - \mathcal U}_a
\leq
\int_\Omega \norm{U_{N,\omega} - U_\omega}_a\, d\mathbb P(\omega).
\]

\medskip

\noindent
Step 1: Pointwise convergence.
Let $f = \sum_n f_n e_n$ with $\|f\|_a=1$. Then one can write 
\(
f= \Pi_N f + (I-\Pi_N)f,
\)
where, by definition, \(\Pi_N f= \sum_{|n|\le N} f_n e_n\) and \((I-\Pi_N)f= \sum_{|n|>N} f_n e_n\). 

Accordingly, 
\begin{align*}
    (U_\omega-\Pi_N\circ U_\omega\circ \Pi_N)f &= U_\omega(I-\Pi_N) f + (I-\Pi_N) U_\omega\Pi_Nf\\
    &= U_\omega \sum_{|n|>N} f_n e_n+(I-\Pi_N)U_\omega \sum_{|n|\le N} f_n e_n.
\end{align*}

For the first term, using \cite[Lemma~2.4]{SBJ17}, 
for every \(\omega\in\Omega\), there exists \(0<\delta_\omega<1\), \(C_\omega>0\) such that 
\[
\norm{U_\omega(e_n)}_a \le C_\omega e^{-\delta_\omega |n|},
\]
which implies \(U_w\) is a Hilbert-Schmidt operator and 
\[
\norm{U_\omega \sum_{|n|>N} f_n e_n}_a \le
\sum_{|n|>N} |f_n| \norm{U_\omega e_n}_a \le
C_\omega \sum_{|n|>N} |f_n| e^{-\delta_\omega |n|}
\to 0,
\]
since \((e_n)\) is orthonormal in \(H_a\), we have \(\norm{f}^2_a=\sum_n|f_n|^2<\infty\). Then Cauchy-Schwarz gives
\[
\sum_{n>N}|f_n| e^{-\delta_\omega |n|}\le \left(\sum_{|n|>N}|f_n|^2\right)^{1/2}\left(\sum_{|n|>N}e^{-2\delta_\omega |n|}\right)^{1/2}
\]
which tends to 0.

For the second term, write 
\[
(I-\Pi_N)U_\omega \Big(\sum_{|n|\le N} f_n e_n\Big) =
\sum_{|n|\le N} f_n (I-\Pi_N)U_\omega e_n.
\]

Hence
\(
\norm{(I-\Pi_N)U_\omega \Pi_N f}_a \le
\sum_{|n|\le N} |f_n| \sum_{|k|>N} |\langle U_\omega e_n, e_k\rangle|.
\)
Fix \(\omega\), since \(U_\omega\) is Hilbert-Schmidt, its matrix coefficients 
\(
\Gamma_{k,n}(\omega):=\langle U_\omega e_n, e_k\rangle
\)
satisfy 
\(
\sum_{k,n \in \mathbb Z} |\Gamma_{k,n}(\omega)|^2 < \infty.
\)
Hence,  \(\sum_{|k|>N,n \in \mathbb Z} |\Gamma_{k,n}(\omega)|^2\to 0\) as \(N\to\infty\). 
This implies that\footnote{\(\norm{\cdot}_{HS}\) is the Hilbert-Schmidt norm given by \(\langle A,B\rangle_{HS}= tr(B^*A)\) where \(tr(A)\) and \(A^*\) are respectively the trace and adjoint of the matrix \(A\).} 
\[
\norm{(I-\Pi_N)U_\omega \Pi_N}_{HS}= \sum_{|k|>N,\\
|n| \le N} |\Gamma_{k,n}(\omega)|^2\to 0\ \text{as}\ N\to\infty.
\]
Since \(\norm{(I-\Pi_N)U_\omega \Pi_N}_a\le \norm{(I-\Pi_N)U_\omega \Pi_N}_{HS}\), we get
 \(\norm{U_{N,\omega} - U_\omega}_a \to 0\) for all \( \omega\in\Omega\).\\

\noindent
Step 2: Dominated convergence.
Since $\|\Pi_N\|=1$, we have  \( \norm{U_{N,\omega}}_a \le \norm{U_\omega}_a\), 
hence \( \norm{U_{N,\omega} - U_\omega}_a \leq 2\norm{U_\omega}_a\).

Since, $2\norm{U_\omega}_a \in L^1(\Omega)$, therefore, by dominated convergence,
\[
\int_\Omega \norm{U_{N,\omega} - U_\omega}_a\, d\mathbb P(\omega) \to 0.
\]

\medskip

Combining the above estimates yields that as \(N\to\infty\), \(\norm{\mathcal U_N - \mathcal U}_a \to 0\).
\end{proof}

\begin{remark}\label{rem:diagonal}
    Using \cite[Lemma~3.1]{SBJ17} and its proof, we get that the matrix given by \eqref{eq:Gamma} is lower-triangular with respect to the basis re-ordered as in formula \eqref{eq:basis}. The only non-zero diagonal entries of this matrix are 
    \[
\Gamma_{00,00} = 1, \quad
\Gamma_{0k,0k} = \int (-\omega)^kd\bP(\omega), \quad
\Gamma_{0{-k},0{-k}} = \int (-\bar{\omega})^kd\bP(\omega),
\]
that is the only non-zero diagonal entries occur in the following cases  
\begin{itemize}
    \item \(n_1=m_1=n_2=m_2=0\) or
    \item \(n_1=m_1=0\) and \(n_2=m_2=k>0\) or
    \item \(n_1=m_1=0\) and \(n_2=m_2=k<0\).
\end{itemize}
Furthermore, if we divide the matrix in four blocks with the upper right one corresponding to the entries with \(n_1n_2\geq0\), \(m_1m_2<0\) consists of all zeros; the upper left one with \(n_1n_2<0\), \(m_1m_2<0\) is a lower triangular block with zeros on the diagonal, so the only interesting block is the lower right one with \(n_1n_2\geq0\) and \(m_1m_2\geq 0\) which is also lower triangular and has non-zero diagonal entries.
\end{remark}

\begin{proof}[Proof of Theorem~\ref{thm:Bl_spectrum}]
 
 Compactness of  \(\cU: H_a \to H_a\) was established in Remark~\ref{rem:ann_compact}. By Lemma~\ref{lem:finite-approx}, \(\cU_N\) is the finite approximation  of \(\cU\), therefore the matrix for \(\cU_N\) is the same as given by \eqref{eq:Gamma}, and thus using Remark~\ref{rem:diagonal}, we get that this matrix is lower-triangular with respect to the basis in \eqref{eq:basis}. 
 Moreover, the only non-zero diagonal entries of this matrix are given as follows, for \(k\in \bN\): 
 \[
 \Gamma_{00,00}= \int_\Omega1\ d\bP(\omega) = 1,\ \Gamma_{0k,0k} = \int_\Omega (-\omega)^k\ d\bP(\omega)\  \text{and}\  \Gamma_{0-k,0-k} = \int_\Omega (-\bar\omega)^k\ d\bP(\omega).
 \]
 Therefore, the spectrum of \(\cU_N\), for \(k\in\bN\), is given by 
     \begin{equation}\label{eq:U_n-spectrum}
         \sigma(\cU_N) = \left\{\int_\Omega (-\omega)^k\ d\bP(\omega)\right\} \cup \left\{\int_\Omega (-\bar\omega)^k\ d\bP(\omega)\right\} \cup \{1,0\}.
     \end{equation}
     By definition, each non-zero value that occurs in \eqref{eq:U_n-spectrum}, is an eigenvalue with the algebraic multiplicity equal to the number of times of its occurrence. Again, by Remark~\ref{rem:diagonal}, we know that the above matrix is lower-triangular with each row either consisting of all zeros or exactly one non-zero entry (which is on the diagonal if it is the \(0k\)th row, that is the diagonal entry has the address \((0k,0k)\) and off-diagonal otherwise). Hence, every non-zero eigenvalue has its algebraic multiplicity equal to its geometric multiplicity. Indeed,  
     for the eigenvalue 1, both the geometric and algebraic multiplicity are 1. Let \(\lambda_k\in \left\{\int_\Omega (-\omega)^k\ d\bP(\omega)\right\} \cup \left\{\int_\Omega (-\bar\omega)^k\ d\bP(\omega)\right\}\) be a non-zero eigenvalue lying in the \(k\)th row, then the basis vector \(e_{0,k}\) serves as an eigenvector for \(\lambda_k\). If there exists \(l(\neq k)\in\bN\) such that \(\lambda_k=\lambda_l\neq 0\), then the corresponding eigenspace has dimension at least 2, because the vectors \(e_{0,k}\) and \(e_{0,l}\) both lie in the eigenspace, making the geometric and algebraic multiplicity equal. 

     Finally, it suffices to justify that the non-zero spectrum with the respective algebraic and geometric multiplicities, of the operator $\cU$, is captured by the non-zero spectra of the finite rank operators \(\cU_N\). This follows from a standard spectral approximation result (see, for example, \cite[XI.9.5]{DunSch63}) together with the fact that \(\cU_N\) converges to \(\cU\) in the operator norm on \(H_a\) (see Lemma~\ref{lem:finite-approx}), and that \(\cU\) is compact (see Remark \ref{rem:ann_compact}).
 \end{proof}

 \subsection{Resonances of a compact annealed operator associated to an Stochastic differential equations }\label{sec:sde}
In this subsection, we present an example numerically approximating the spectrum of a compact annealed transfer operator. This example has been explored in \cite{DGJ25} to numerically compute the optimal perturbation for a given observable. We use their example and the operator they have approximated to compute its spectrum.\\
Consider, for a noise intensity $\varepsilon>0$ and a final time $T>0$, the SDE given by
\begin{equation}
\begin{split}
    dY_t^x &= - V'(Y_t^x)dt + \varepsilon dW_t, \quad t \in (0,T)\\
    Y_0^x &= x,
\end{split}
\label{eq: SDE gradient type}
\end{equation}
where $V(y) = \frac{y^4}{4}- \frac{y^2}{2}$ is a symmetric double well potential. $(W_t)_t$ denotes a Brownian Motion (BM) and $x \in \mathbb{R}$ is a deterministic initial condition.

As in \cite{DGJ25}, for the final time \(T>0\), the kernel of the associated transfer operator is given by  $\kappa(x,y) := p^x(y,T)$ (which is actually the transition density of the SDE), the numerical approximation is displayed in Figure~\ref{subfig: kernel} and the transfer operator is given by 
\[
(\cL f)(y)=\int \kappa(x,y) f(x)\ dx. 
\]
Building on the code provided in \cite{DGJ25} for approximating the kernel and the matrix associated to the transfer operator, we compute the invariant density (see Figure~\ref{subfig: invarian density f_0}). Finally, we numerically approximate the eigenvalues of the matrix/Ruelle-Pollicott resonances of the system in \eqref{eq: SDE gradient type} as given in Figure~\ref{subfig:resonances}.\\
The MATLAB code used to generate the numerical results is publicly available on Zenodo at \url{https://doi.org/10.5281/zenodo.21816093}.

\begin{figure}\label{fig:SDE}
\begin{subfigure}[h]{0.30\linewidth}
\includegraphics[width=\linewidth]{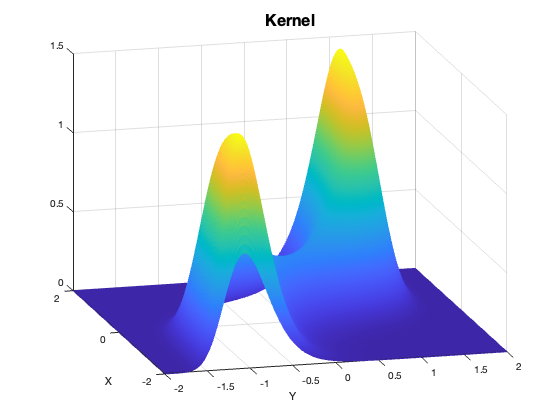}
\caption{Kernel $(x,y) \mapsto \kappa(x,y)$}
\label{subfig: kernel}
\end{subfigure}
\hfill
\begin{subfigure}[h]{0.30\linewidth}
\includegraphics[width=\linewidth]{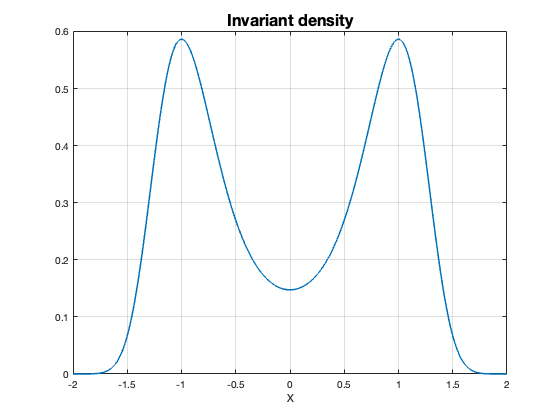}
\caption{Invariant density for $\mathcal{L}_0$}
\label{subfig: invarian density f_0}
\end{subfigure}%
\hfill
\centering
\begin{subfigure}[h]{0.30\linewidth}
    \includegraphics[width=\linewidth]{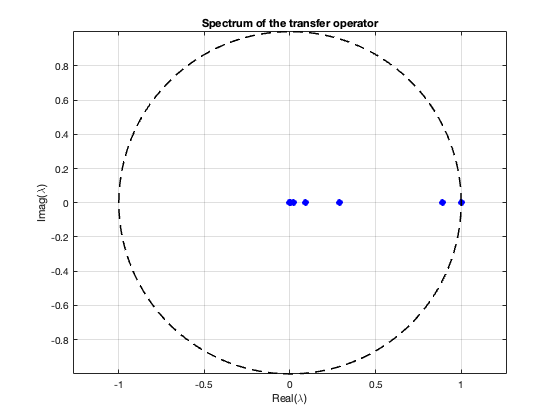}
    \caption{Ruelle-Pollicott resonances}
    \label{subfig:resonances}
\end{subfigure}
\caption{Numerical approximation for (a) kernel $\kappa(x,y)$, (b) invariant density for the transfer operator $\mathcal{L}_0$ and (c) The Ruelle-Pollicott resonances of the system. Here the final time is $T =1,$ the noise intensity is $\varepsilon = 0.6$, the domain is $\Omega = (-2,2)$, while the mesh sizes are $\Delta x = \Delta x = 2\cdot 10^{-3}$.}
\end{figure}

\appendix
\section{Proofs of Framework results from \ref{subseciidintro}.}\label{sec:classic}

We start with this Appendix with the following general definition of the transfer operator $\cL \in L(\cB)$ introduced in equation \eqref{eqoptrans} whenever $\cB\subset L^1(m)$ is a Banach space.

\begin{proposition}\label{prop:integomeg}
Let $\cB\subset L^1(m)$ be a Banach space and assume that $\omega \in \Omega \mapsto L_\omega \in L(\cB)$ is measurable and that $\int \|L_\omega\|_\cB d\bP(\omega)<\infty$. Let $f\in \cB$ then for any measurable set $A \subset \Omega$, there is an element $e_A(f)\in \cB$ with $\|e_A\|_\cB\leq \int_A \|L_\omega\|_\cB d\bP(\omega)$ such that for any $g\in \cB^*$,
\begin{align}\label{eq:weakint}
\langle e_A(f),g \rangle_{\cB,\cB^*} =\int_A \langle  L_\omega f,g \rangle_{\cB,\cB^*} d\bP(\omega).
\end{align}

Furthermore $\Psi : f\in \cB \mapsto e_A(f)\in \cB$ is a linear operator bounded from above by
\begin{align}\label{eq:conL}
    \|\Psi\|_{L(\cB)}\leq \int_A\|L_\omega\|_{L(\cB)}d\bP(\omega)
\end{align}

\end{proposition}
\begin{remark}
    That element $e_A$ is denoted by $e_A=\int_A L_\omega d\bP(\omega)$, represents what we formally introduced in subsection \ref{subseciidintro} and satisfies the same classical properties of additivity and positiveness of integrals.\\
    An example of invariant Banach space $\cB$ invariant for $\cL$ is given by $\cB:=L^\infty(m)$ whenever $L^1(m)$ is invariant by each $U_\omega$.
\end{remark}

\begin{proof}
Let $J : e\in \cB\mapsto \langle e,\cdot\rangle_{\cB,\cB^*}\in \cB^{**}$ be the standard isometry and let $\bP^A$ be the conditional probability with respect to $A$. The random variable $Y : \omega\in \Omega \mapsto L_\omega$ generates a push forward probability law $\bP^A_Y$ over $\cB^{**}$. Furthermore, 
for $n\in \bN$, there is a compact set $\tilde K_n$ such that $\bP^A_Y(\tilde K_n)>1-\frac 1 n$. For any $\ve>0$, using dominated convergence, one can prove that there is a compact set $K_\ve$ such that for any $g\in \cB^*$, 
\begin{align}\label{eq:intcompo}
    \int_{A\cap (Y^{-1}(K_\ve))^c} \|L_\omega f\|_{\cB}d\bP(\omega)\leq \ve.
\end{align}
\footnote{where \(M^c\) denotes the complement of the set \(M\).}Thus, for a small enough $\eta_\ve>0$, taking a partition $(F_i^{\ve})_{i\in \bN}$ from a finite covering of the compact $K_\ve$ by balls of radius $\eta_\ve$ in $(L(\cB),\norm{\cdot}_{L(\cB)})$, one has for any $g\in \cB^*$, 
\begin{align}\label{eq:intosumlinform}
    &\left|\int_A\langle  L_\omega f,g \rangle_{\cB,\cB^*}\ d\bP -\left(\int_{Y^{-1}(K_\ve^c)}\langle  L_\omega f,g \rangle_{\cB,\cB^*} d\bP+\bP(A)\sum_i \langle  L_i f,g \rangle_{\cB,\cB^*} \bP^A_Y(F_i^{\ve})\right)\right| \nonumber\\
    &\leq \|g\|_{\cB^*}\bP(A)\sum_i \int_{F_i^{\ve} } \left\|L_\omega-L_i\right\|_{\cB}d\bP^A_Y(\omega) \norm{f}_{\cB}\nonumber\\
    &\leq \ve \|g\|_{\cB^*}\norm{f}_{\cB} .
\end{align}

Thus choosing a sequence of terms $\ve_n>0$ such that $\ve_n \underset{n\to \infty}{\to}0$, one has from equation \eqref{eq:intosumlinform} and \eqref{eq:intcompo},

\begin{align*}
    \left\| \int_{A} \langle  L_\omega f,\cdot \rangle_{\cB,\cB^*} d\bP(\omega)-\langle \bP(A)\sum_i \bP^A_Y(F_i^{\ve_n})  L_i f,\cdot \rangle_{\cB,\cB^*}  \right\|_{\cB^**}\to 0. 
\end{align*}
In other words the sequence $\left( \bP(A)\sum_i \bP^A_Y(F_i^{\ve_n})  L_i f\right)_{n\in \bN}\in J(\cB^*)^{\bN}$ converges to $\int_{A} \langle  L_\omega f,\cdot \rangle_{\cB,\cB^*} d\bP(\omega)$ and since $J(\cB)$ is Banach space, the latter linear form admits some Riesz representation in $\cB$. That representation is the element $\int_{A}   L_\omega fd\bP(\omega) \in \cB$. The linearity of $f\mapsto e_A(f)$ is a natural consequence of the relation \eqref{eq:weakint} that defines it.

The last point \eqref{eq:conL}, follows from the previous asserted relation \eqref{eq:weakint} :
\begin{align*}
    \left\|\int_AL_\omega f d\bP(\omega)\right\|_{\cB}&=\left\|\langle \int_AL_\omega f d\bP(\omega),\cdot\rangle_{\cB,\cB^*}\right\|_{\cB^*}\\
    &=\left\|\int_A \langle L_\omega f,\cdot\rangle_{\cB,\cB^*} d\bP(\omega)\right\|_{\cB^*}\\
    &\leq \int_A\left\| L_\omega  f \right\|_{\cB} d\bP(\omega)\\
    &\leq \int_A\left\| L_\omega \right\|_{L(\cB)}d\bP(\omega)\left\| f \right\|_{\cB}.
\end{align*}

\end{proof}
The following is an intermediate lemma useful for establishing the measurability of the map \(\omega\mapsto L_\omega\). 
\begin{lemma}\label{lemmesurabl}
Let $\cB$ be a separable Banach space and suppose that there is a dense family $\{f_i\in \cB,i\in \bN\}$ such that for any $f_i \in \cB$, $\omega \mapsto L_\omega f_i \in \cB$ is measurable in $\cB$, then $\omega \mapsto L_\omega(\cdot)\in L(\cB)$ is measurable.
\end{lemma}

\begin{proof}
   Let $\ve>0$ and $(\ve_n)_{n\in \bN}$ an increasing sequence such that $\ve_n\to \ve$. We first prove that $\{\omega: L_\omega\in B_{L(\cB)}(0,\ve) \}$ is measurable, for any ball $B_{L(\cB)}(0,\ve)$ of size at most $\ve$ for the operator norm on $L(\cB)$.
   Let $\Psi : \omega \mapsto ((f\mapsto L_\omega f) \in L(\cB))$. Recall that by assumption, for any $f_i$, $\psi_{f_i}: \omega \mapsto L_\omega f_i \in \cB$ is measurable and since $\cB$ is separable, $\Psi^{-1}(B_{L(\cB)}(0,\ve))=\bigcup_{n\in \bN}\bigcap_{i\in \bN}(\psi_{f_i}^{-1}(B_{\cB}(0,\ve_n)))$ is measurable. Thus, any open ball in the operator norm topology are measurable and the map $\Psi$ is a measurable map on $L(\cB)$ equipped with the Borel sigma-algebra of $L(\cB)$, that is $\omega \mapsto L_\omega \in L(\cB)$ is measurable.
\end{proof}

We introduce the following Lemma linking the spectral properties of Transfer operators with decorrelation formulae :

\begin{lemma}\label{lem:equivalence}
    Let $\cL\in L(\cB)$ be a quasi-compact operator on $\cB$ such that $\sigma(\cL)\backslash B(0,\ve)=\{\lambda_1,\dots,\lambda_M\}$ where the $\lambda_i$ are of finite multiplicities (geometric, in other words $\sigma_{ess}(\cL) \subset B(0,\ve)$), 
then for $i\leq  M$ there are quantities $m_i\in \bN$, the geometric multiplicity of $\lambda_i$, such that for any  $k\leq m_i$, there are bilinear map $C'_{i,k} : \cB\times \cB\mapsto \mathbb{C}$ such that

\begin{align}
    \int_I \cL^n\phi \varphi dm =\ \sum\limits_{i=1}^{ M}\sum\limits_{k=0}^{m_i-1}\lambda_i^n n^k C'_{i,k}(\phi,\varphi) + o(\ve^n)
\end{align}
\end{lemma}

\begin{proof}[Proof of Lemma~\ref{lem:equivalence}]

Since the
transfer operator is quasi-compact such that there is $M>0$ and spectral projectors $N_\ve, \Pi_1,\dots,\Pi_M\in L(B)$ 
 \begin{align}\label{eqquasicompdeco}
\cL= \sum_{i=1}^{M}\cL\circ\Pi_i +N_\ve    
\end{align}
    with $\norm{N_\ve^n}_{L(\cB)}\leq C\ve^n$ and $ \Pi_i$ are finite rank spectral projectors associated to the eigenvalue $\lambda_i$ orthogonal to $N_\ve$ and to one another, that is, \(\Pi_i \circ \Pi_j=0\).

Since the operators $\cL\circ \Pi_i$ are of finite rank, it acts on a finite dimensional subspace, and therefore can be decomposed into a triagonal Jordan form, that is  $\cL\circ\Pi_i:=\lambda_i\left( \Pi_i+Q_i\right)$ with $m_i$ being the geometric multiplicity of $\lambda_i$, $Q_i$ a nilpotent operator with $Q_i^{m_i}=0$ and $\Pi_i$ match with the diagonal part, in particular projector commutes with $Q_i$ : $ \Pi_i\circ Q_i=Q_i\circ \Pi_i=Q_i$. Thus, 
    \begin{align}\label{eqquasicompact}
\cL= \sum_{i=1}^{M}\lambda_i(  \Pi_i+Q_i) +N_\ve    
\end{align}

where $\norm{N_\ve^n }=O(\ve^n)$ an operator orthogonal to  $ \Pi_i$ and $Q_i$, and $ \Pi_kQ_i=Q_i \Pi_k=\delta_{i,k}Q_i$.
In particular,
$$
\cL^n=\sum_{i=1}^{M}\lambda^n_i (\Pi_i+Q_i)^n +N_\ve^n=\sum_{i=1}^{M}\sum_{ 0\leq k \leq m_i} {n \choose k}\lambda_i^{n} Q_i^{k}\circ \Pi_i +N_\ve^n=\sum_{i=1}^{M}\sum_{0\leq k \leq m_i-1}n^k\lambda_i^{n} \tilde C_{i,k}Q_i^{k}\circ \Pi_i +N_\ve^n
$$

 We thus obtain
\begin{align}
    \int \phi \cU^nf dm &=\int \cL^n\phi f dm\\
    &=\int f\sum_{i=1}^{M}\sum_{1\leq k \leq m_i-1}n^k\lambda_i^{n} \tilde C_{i,k}Q_i^{k}\phi dm +\int fN_\ve^n\phi dm\\
    &=\sum_{i=1}^{M}\sum_{1\leq k \leq m_i-1}\lambda_i^nn^k C_{k,i}(\phi,f)+o(\eps^n),
\end{align}
with $C_{i,k}(\phi,f):=\int f \tilde C_{i,k}Q_i^{k}\phi dm $
\end{proof}

\begin{proof}[Proof of Proposition~\ref{prop:gencomp}]
Let $Y: \omega \in \Omega \mapsto L_\omega \in L(\cB)$, according to Lemma \ref{lemmesurabl}, $Y$ is measurable, thus $\mu:=Y*\bP$ is a probability measure on $L(\cB)$. Then for any $\ve>0$ there is a compact set $K_\ve$ such that $\mu(K_\ve)\geq 1-\ve$ and $\int_{Y^{-1}(K_\ve^c)} \norm{L_\omega}\ d\bP(\omega)\leq \ve$. Thus for a small enough $\eta_\ve>0$, taking a covering of $K_\ve$ by balls $F_i$ of radius $\eta_\ve$ in $(L(\cB),\norm{\cdot}_{L(\cB)})$, we deduce from the inequality \eqref{eq:conL}, 
\begin{align}\label{eq:intosum}
    \norm{\int L_\omega (\cdot)\ d\bP -\left(\int_{Y^{-1}(K_\ve^c)}L_\omega (\cdot)\ d\bP+\sum_i L_i(\cdot) \mu(F_i)\right)}_{L(\cB)} \leq \ve.
\end{align}
where we pick for $L_i$ one of the $L_\omega$ (one belonging to the same set $F_i$), the sum $\sum_i L_i \mu(F_i)$ is a compact operator being a finite sum of compact operators and finally $\int L_\omega(\cdot)\ d\bP$ is a compact operator being a limit of compact operators (by closure of the space of compact operators).
\end{proof}

\begin{proof}[Proof of Proposition~\ref{propgenquasi}]
    We know from Proposition \ref{prop:gencomp} that $\int \Pi_{\omega,r}d\bP(\omega)$ is a compact operator and thus 
    \begin{align*}
\norm{\cL-\int \Pi_{\omega,r}d\bP(\omega)}_{W_r}&\leq \int\norm{N_{\omega,r}}_{W_r}d\bP(\omega)\\
&\leq \int \ve_r(\omega)d\bP(\omega).
    \end{align*}
Since $\bP(\{\omega,\lim_{r\to \infty}\ve_r(\omega)=0\})=1$ and $\ve_r(\omega)< 1$, for $\ve>0$, there is $r_0>0$ such that for $r\geq r_0$, $\int \ve_r(\omega)d\bP(\omega)\leq \ve$. Thus the essential spectral radius $r_{ess}(\cL)$ satisfies $r_{ess}(\cL)\leq \ve$ which ensures that the operator is quasi compact with the following decomposition
\begin{align}\label{eq:decospecgap}
     \cL=\Pi_{r}+N_{r}
\end{align}
    
as required.\\

Thanks to Lemma \ref{lem:equivalence}, \eqref{eq:decospecgap} implies that the decorrelation property holds for any $r\geq r_0$. That is, for $\lambda_1>\dots>\lambda_n>\dots$ the eigenvalues\footnote{the enumeration is made possible by the fact that for any $\ve>0$ the eigenvalues of $\cL$ on $\bigcap_rW_r$ such that $\lambda>\ve$ belongs to the finite family of eigenvalues for $P$ on $W_r$ that are above $\ve_r>0$ for some $r$ such that $\ve_r\leq \ve$.} of $\cL$ over $\bigcap_rW_r$ such that $\lambda_i>\ve_{r_0}$ there are bilinear form $C^r_{i,k} :\bigcap_{r>0}W_r\times \bigcap_{r>0}W_r\mapsto \bC$ and integers $(m_i)_{i\geq 0}$ such that
\begin{align}
    \int_I \cL^n\phi \varphi dm =\ \sum\limits_{i=1}^{ M}\sum\limits_{k=0}^{m_i-1}\lambda_i^n n^k C^{r}_{i,k}(\phi,\varphi) + o(\ve^n).
\end{align}
To get the \textbf{full Ruelle-Pollicott resonances spectrum}, it remains to show that the bilinear map  $C^r_{i,k}$ actually do not depend on the choice of $r\in \bN$. To do so, recall that, according to the proof of Lemma \ref{lem:equivalence}, $C^r_{i,k}$ derives from the operator $\cL\circ \Pi_r(\lambda_i)$ where  $\Pi_r(\lambda_i)$ is  the spectral projector associated to the general eigenspace $E_r(\lambda_i)$. Thus it is enough to show that for any $f\in \bigcap_{r>0}W_r$,
\begin{align}
    \cL\circ \Pi_r(\lambda_i)(f)=\cL\circ \Pi_{r_0}(\lambda_i)(f),\, \forall\ r\geq r_0.
\end{align}
For the sake of completeness, we rigorously check that $\Pi_r(\lambda_i)=\Pi_{r_0}(\lambda_i)$ over $\bigcap_{r>0}W_r$.
Recall that (see \cite{DunSch63}) for $r\geq r_0$,
\begin{align*}
    \Pi_r(\lambda_i)=\frac 1 {2i\pi}\int_{\Gamma} f(z)(z-\cL_{|W_r})^{-1}dz
\end{align*}
for some closed curve $\Gamma \subset \bC \backslash \left\{B(0,\eps_r)\cup_{i\geq 0}\{\lambda_i\}\right\}$ encompassing $\lambda_i$. 
Since any $(z-\cL_{|W_r})^{-1}$ extends uniquely continuously into $(z-\cL_{|W_{r_0}})^{-1}$ on $W_{r_0}$, then $\Pi_{r}(\lambda_i)$ extends continuously to $\Pi_{r_0}(\lambda_i)$ over $(W_{r_0},\|\cdot\|_{r})$. Thus both projectors coincide over $\bigcap_rW_r$ and so does the bilinear form $C^{r}_{i,k}$ and $C^{r_0}_{i,k}$. Thus there are eigenvalues $\lambda_1>\dots>\lambda_n>\dots$ of $\cL_{|\bigcap_{r\geq 0}W_r}$, quantities $m_i\in \bN$ corresponding to the geometric multiplicity of $\lambda_i$ such that for any  $k\leq m_i$ and bilinear map $C'_{i,k} : \cB\times \cB\mapsto \mathbb{C}$ such that for any $\eps>0$,

\begin{align}
    \int_I \cL^n\phi \varphi dm =\ \sum_{i,\lambda_i>\ve}\sum\limits_{k=0}^{m_i-1}\lambda_i^n n^k C'_{i,k}(\phi,\varphi) + o(\ve^n),\; \forall \phi,\psi\in \bigcap_rW_r.
\end{align}
\end{proof}

{\bf Data Availability.} {\small No datasets were generated or analyzed during the current study. Code and numerical simulation used for the illustrative example is available at  \url{ https://doi.org/10.5281/zenodo.21816093}.}

{\bf Acknowledgments.} {\small The authors would like to express their gratitude to Carlangelo Liverani for an insightful initial discussion and to Françoise Pène for her constructive suggestions.
Part of this project was conducted while SJ was affiliated with Monash University whose support is gratefully acknowledged. An ARC Laureate Fellowship FL230100088 fully supported the research of MP.}

\clearpage

\bibliographystyle{plain}
 \bibliography{Biblio/biblio}

\end{document}